\pdfoutput=1
\documentclass{amsart}

\usepackage[inline]{asymptote}
\usepackage{amssymb,mathtools,nicefrac}
\usepackage{hyperref}
\usepackage[lite]{amsrefs}
\usepackage{enumitem}
\usepackage{microtype}

\newtheorem{theorem}{Theorem}[section]
\newtheorem{corollary}[theorem]{Corollary}
\newtheorem{lemma}[theorem]{Lemma}
\newtheorem{proposition}[theorem]{Proposition}

\theoremstyle{remark}
\newtheorem{example}[theorem]{Example}
\newtheorem{remark}[theorem]{Remark}
\theoremstyle{definition}
\newtheorem{definition}[theorem]{Definition}

\DeclareMathOperator{\Aut}{Aut}
\DeclareMathOperator{\Endo}{End}
\DeclareMathOperator{\Hom}{Hom}
\DeclareMathOperator{\Ind}{Ind}
\DeclareMathOperator{\cInd}{cInd}
\DeclareMathOperator{\Res}{Res}
\DeclareMathOperator{\Star}{Star}
\DeclareMathOperator{\tr}{tr}
\DeclareMathOperator{\im}{im}
\DeclareMathOperator{\supp}{supp}
\DeclareMathOperator{\Smooth}{S}
\DeclareMathOperator{\Rough}{R}

\newcommand*{\nb}{\nobreakdash}
\newcommand*{\defeq}{\mathrel{\vcentcolon=}}
\newcommand*{\congto}{\xrightarrow\cong}
\newcommand*{\mono}{\rightarrowtail}
\newcommand*{\epi}{\twoheadrightarrow}

\newcommand*{\abs}[1]{\lvert#1\rvert}
\newcommand*{\realise}[1]{\lvert#1\rvert}
\newcommand*{\av}[1]{\langle#1\rangle}

\newcommand{\twolinesubscript}[2]{\genfrac{}{}{0pt}{}{#1}{#2}}

\newcommand*{\N}{\mathbb N}
\newcommand*{\Z}{\mathbb Z}
\newcommand*{\Q}{\mathbb Q}
\newcommand*{\R}{\mathbb R}
\newcommand*{\C}{\mathbb C}
\newcommand*{\Mat}{\mathbb M}
\newcommand*{\Field}{\mathbb F}
\newcommand*{\Laurent}[1]{\mathbb F_{#1}\mathopen{[\![}t,t^{-1}]}

\newcommand*{\Id}{\textup{Id}}
\newcommand*{\diff}{\textup d}
\newcommand*{\fin}{\textup f}
\newcommand*{\op}{\textup{op}}

\newcommand*{\Proj}{e}
\newcommand*{\support}{u}
\newcommand*{\Hecke}{\mathcal H}
\newcommand*{\Mult}{\mathcal M}
\newcommand*{\Rep}{\mathfrak{Rep}}
\newcommand*{\Subrep}{\mathcal S}
\newcommand*{\Abcat}{\mathcal C}

\newcommand*{\Ho}{\textup H}
\newcommand*{\Cell}[1][*]{C_{#1}}
\newcommand*{\DCell}[1][*]{C^{#1}}
\newcommand*{\Hull}{\mathcal H}

\newcommand*{\LF}{\mathbb K}
\newcommand*{\inv}{\times}
\newcommand*{\integ}{\mathcal O}
\newcommand*{\maxid}{\mathcal P}
\newcommand*{\unif}{\varpi}

\newcommand*{\G}{\mathcal G}
\newcommand*{\CG}{K}
\newcommand*{\ses}{\textup{ss}}
\newcommand*{\Gl}{\textup{Gl}}

\newcommand*{\Buil}{\mathcal{BT}(\G_\LF)}
\newcommand*{\BuilGl}{\mathcal{BT}}
\newcommand*{\Apart}{A}
\newcommand*{\affRoot}{a} 
\newcommand*{\varaffRoot}{\tilde a} 
\newcommand*{\StabP}{P} 
\newcommand*{\StabS}{P^\dagger} 

\begin{document}
\begin{asydef}
defaultpen(font("OT1","cmr","m","n")+fontsize(10));
int i;//generic counter
pair a[];//the three unit directions of an Ã2 building
a[0]=2*(1,0);
a[1]=2*dir(60);
a[2]=2*dir(120);
\end{asydef}

\title[Buildings and resolutions of representations of \(p\)-adic groups]{Resolutions for representations of reductive \(p\)-adic groups via their buildings}
\author{Ralf Meyer}
\email{rameyer@uni-math.gwdg.de}
\author{Maarten Solleveld}
\email{maarten@uni-math.gwdg.de}

\address{Mathematisches Institut and\\
  Courant Centre ``Higher order structures''\\
  Georg-August Universit\"at G\"ottingen\\
  Bunsenstra{\ss}e 3--5\\
  37073 G\"ottingen\\
  Germany}

\begin{abstract}
  Schneider--Stuhler and Vign\'eras have used cosheaves on the affine Bruhat--Tits building to construct natural projective resolutions of finite type for admissible representations of reductive \(p\)\nb-adic groups in characteristic not equal to~\(p\).  We use a system of idempotent endomorphisms of a representation with certain properties to construct a cosheaf and a sheaf on the building and to establish that these are acyclic and compute homology and cohomology with these coefficients.  This implies Bernstein's result that certain subcategories of the category of representations are Serre subcategories.  Furthermore, we also get results for convex subcomplexes of the building.  Following work of Korman, this leads to trace formulas for admissible representations.
\end{abstract}
\thanks{Supported by the German Research Foundation (Deutsche Forschungsgemeinschaft (DFG)) through the Institutional Strategy of the University of G\"ottingen.}
\maketitle
\tableofcontents

\section{Introduction}
\label{sec:intro}

Let~\(\G_\LF\) be a reductive \(p\)\nb-adic group, that is, the group of \(\LF\)\nb-rational points of a reductive linear algebraic group over a non-Archimedean local field~\(\LF\).  Let~\(\Buil\) be the affine Bruhat--Tits building of~\(\G_\LF\).  Work related to the Baum--Connes conjecture has shown that \(\Buil\) knows a lot about topological properties of the category of smooth representations of~\(\G_\LF\) (see~\cites{Baum-Higson-Plymen:Representation, Solleveld:Periodic_Hecke_Schwartz}).  This article follows earlier work by Peter Schneider and Ulrich Stuhler~\cite{Schneider-Stuhler:Rep_sheaves} who use the building to construct natural projective resolutions of finite type for admissible representations of~\(\G_\LF\).  This was extended by Marie-France Vign\'eras~\cite{Vigneras:Cohomology} to representations on vector spaces over fields of characteristic not equal to~\(p\).

These resolutions may be used to study Euler--Poincar\'e functions of representations, to compute formal dimensions of discrete series representations, and to compute the inverse of the Baum--Connes assembly map on the K\nb-theory classes of discrete series representations (see \cites{Schneider-Stuhler:Rep_sheaves, Meyer:Homological_Schwartz}).

The input data for the resolutions of Schneider--Stuhler, besides a representation \(\pi \colon \G_\LF\to\Aut(V)\), is a carefully chosen system of compact open subgroups~\(\CG_\sigma\) for all polysimplices~\(\sigma\) in \(\Buil\), depending on a parameter \(e\in\N\).  Let~\(V_\sigma\) be the subspace of \(\CG_\sigma\)\nb-fixed points in~\(V\).  Then \((V_\sigma)_{\sigma\in\Buil}\) is a cosheaf on~\(\Buil\), which gives rise to a cellular chain complex \(\Cell\bigl(\Buil,(V_\sigma)\bigr)\).  This is shown to be a projective resolution of~\(V\) if the subspaces~\(V_x\) for vertices~\(x\) in~\(\Buil\) span~\(V\).  The proof is indirect and depends on Joseph Bernstein's deep theorem that the category of representations~\(V\) that are generated by the subspaces~\(V_x\) is a Serre subcategory in the category of smooth representations of~\(\G_\LF\) (see~\cite{Bernstein:Centre}).

One goal of this article is to obtain a Lefschetz fixed point formula for the character of an admissible representation of~\(\G_\LF\).  This issue was studied by Jonathan Korman in~\cite{Korman:Character}.  He could not get results in the higher rank case because this would require more information about the resolutions of Schneider and Stuhler.  In order to compute the value of the character on a compact regular element~\(g\) of~\(\G_\LF\), we need the cellular chain complex \(\Cell\bigl(\Sigma,(V_\sigma)\bigr)\) to remain acyclic if~\(\Sigma\) is a finite convex subcomplex of~\(\Buil\).  We may choose~\(\Sigma\) invariant under~\(g\), and then the trace of~\(g\) on \(\Cell\bigl(\Sigma,(V_\sigma)\bigr)\) agrees with the trace on~\(V\) for sufficiently large~\(\Sigma\).

We are going to prove directly that \(\Cell\bigl(\Sigma,(V_\sigma)\bigr)\) is a resolution of \(\sum_{x\in\Sigma^\circ} V_x\) for convex subcomplexes \(\Sigma\subseteq\Buil\) and certain cosheaves~\((V_\sigma)\); here~\(\Sigma^\circ\) denotes the set of vertices of~\(\Sigma\).  This implies immediately that the category of representations with \(V=\sum_{x\in\Sigma^\circ} V_x\) is a Serre subcategory of the category of all smooth representations.  Moreover, we can complete Korman's program and formulate a Lefschetz fixed point formula for character values of admissible representations.  We do not yet spend much time to discuss this formula because we hope to establish a more powerful trace formula in a forthcoming article.

The main innovation in this article is the axiomatic formulation of the properties of the cosheaf~\((V_\sigma)\) that are needed for the homology computation.  Our starting point is a system of idempotent endomorphisms \(\Proj_x\colon V\to V\) for vertices~\(x\) in~\(\Buil\) with the following three properties:
\begin{itemize}
\item \(\Proj_x\) and~\(\Proj_y\) commute if \(x\) and~\(y\) are adjacent vertices in~\(\Buil\);
\item \(\Proj_x\Proj_z\Proj_y=\Proj_x\Proj_y\) if \(z\in\Hull(x,y)\), and the vertices \(x\) and~\(z\) are adjacent; here \(x\), \(y\) and~\(z\) are vertices in~\(\Buil\) and~\(\Hull(x,y)\) denotes the intersection of all apartments containing \(x\) and~\(y\);
\item \(\Proj_{gx} = \pi_g\Proj_x\pi_g^{-1}\) for all \(g\in\G_\LF\) and all vertices~\(x\) in~\(\Buil\).
\end{itemize}
Given such a system of idempotents, we let \(\Proj_\sigma\) for a polysimplex~\(\sigma\) in~\(\Buil\) be the product of the commuting idempotents~\(\Proj_x\) for the vertices~\(x\) of~\(\sigma\), and we let \(V_\sigma\defeq \Proj_\sigma(V)\).  This defines a cosheaf on~\(\Buil\), and we show that \(\Cell\bigl(\Sigma,(V_\sigma)\bigr)\) for a convex subcomplex~\(\Sigma\) of~\(\Buil\) is a resolution of \(\sum_{x\in\Sigma^\circ} \Proj_x(V)\), where~\(\Sigma^\circ\) denotes the set of vertices of~\(\Sigma\).

The system~\((\Proj_x)\) provides a sheaf with the same spaces~\(V_\sigma\), using the projections \(\Proj_\sigma\colon V\to V_\sigma\).  We show that the cochain complex \(\DCell\bigl(\Sigma,(V_\sigma)\bigr)\) for this sheaf is a resolution of \(V\bigm/ \bigcap_{x\in\Sigma^\circ} \ker \Proj_x\).  Furthermore, if~\(\Sigma\) is finite then
\[
V\cong \sum_{x\in\Sigma^\circ} \Proj_x(V) \oplus \bigcap_{x\in\Sigma^\circ} \ker \Proj_x.
\]
The idempotent endomorphism~\(\support_\Sigma\) of~\(V\) that effects this decomposition is given by the remarkably simple formula
\[
\support_\Sigma \defeq \sum_{\sigma\in\Sigma} (-1)^{\dim\sigma} \Proj_\sigma.
\]
This fact plays an important role in our proof.

In characteristic~\(0\), the cellular chain complex \(\Cell\bigl(\Buil,(V_\sigma)\bigr)\) consists of finitely generated projective modules if~\(V\) is admissible, so that we get a projective resolution of finite type of \(\sum \Proj_x(V)\).  Since \(V\mapsto \Cell\bigl(\Buil,(V_\sigma)\bigr)\) is an exact functor, the class of representations for which it provides a resolution of~\(V\) is a Serre subcategory.  Thus the class of smooth representations with \(\sum \Proj_x(V)=V\) is a Serre subcategory in the category of all smooth representations of~\(\G_\LF\).  A corresponding statement holds in the cohomological case, provided we use rough representations instead of smooth ones.  By definition, a representation is smooth if it is the \emph{inductive} limit of the subspaces of \(\CG_n\)\nb-invariants, where~\((\CG_n)\) is a decreasing sequence of compact open subgroups with \(\bigcap_{n\in\N} \CG_n=\{1\}\); it is rough if it is the \emph{projective} limit of the same subspaces of \(\CG_n\)\nb-invariants, where we map \(\CG_{n+1}\)\nb-invariants to \(\CG_n\)\nb-invariants by averaging.

Let~\(V\) be an admissible \(\Field\)\nb-linear representation for a field~\(\Field\) whose characteristic is not~\(p\).  Assume \(V=\sum V_x\), and let \(f\colon \G_\LF\to \Field\) be a locally constant function supported in a compact subgroup \(\CG\subseteq\G_\LF\).  Then \(f(V)\) is finite-dimensional and hence contained in \(V|_\Sigma \defeq \sum_{x\in\Sigma^\circ} V_x\) for some finite convex subcomplex~\(\Sigma\) in~\(\Buil\), which we may take \(\CG\)\nb-invariant.  Then \(\Cell\bigl(\Sigma,(V_\sigma)\bigr)\) is a resolution of~\(V|_\Sigma\) by finite-dimensional representations of~\(\CG\).  Hence the character of~\(V|_\Sigma\), restricted to~\(\CG\), is equal to the sum
\[
\chi_\Sigma(g) = \sum_{\twolinesubscript{\sigma\in\Sigma}{g\sigma=\sigma}} (-1)^{\deg \sigma} \chi_{V_\sigma}(g),
\]
where~\(\chi_{V_\sigma}\) denotes the trace of the \(g\)\nb-action on~\(V_\sigma\), with a sign if~\(g\) reverses the orientation of~\(\sigma\).  For the chosen function \(f\in\Hecke(\G_\LF)\), the trace of~\(f\) on~\(V\) agrees with the trace on~\(V|_\Sigma\) because \(f(V)\subseteq V|_\Sigma\).  For arbitrary~\(f\), the trace on~\(V\) will be a limit of such traces on~\(V|_\Sigma\).

The above recipe provides a formula for the values of the character on regular elements.  For regular elliptic elements, this is already contained in~\cite{Schneider-Stuhler:Rep_sheaves}, and for~\(\G_\LF\) of rank~\(1\) such character formulas are established in~\cite{Korman:Character}.

\subsection{Notation and basic setup}
\label{sec:notation_setup}

The following notation will be used throughout this article.

Let~\(\LF\) be a non-Archimedean local field, that is, a finite extension of~\(\Q_p\) for some prime~\(p\) or the field of Laurent series \(\Laurent{q}\) over the finite field~\(\Field_q\) with~\(q\) elements for a prime power~\(q\).  Let~\(p\) be the characteristic of the residue field of~\(\LF\).  Let~\(\integ\) be the maximal compact subring of~\(\LF\) and let~\(\maxid\) be the maximal ideal in~\(\integ\).  Let~\(q\) be the cardinality of the residue field \(\integ/\maxid\).

Let~\(\G\) be a reductive linear algebraic group defined over~\(\LF\).  We write \(\G_\LF\) for its set of \(\LF\)\nb-rational points and briefly call~\(\G_\LF\) a \emph{reductive \(p\)\nb-adic group}.

Recall that~\(\G_\LF\) is a second countable, totally disconnected, locally compact group.  That is, its topology may be defined by a decreasing sequence of compact open subgroups \((\CG_n)_{n\in\N}\).

\subsubsection{Representations as modules over a Hecke algebra}
\label{sec:representations}

Smooth representations of~\(\G_\LF\) on \(\Q\)\nb-vector spaces are equivalent to non-degenerate modules over the Hecke algebra \(\Hecke(\G_\LF,\Q)\) of locally constant, compactly supported \(\Q\)\nb-valued functions of~\(\G_\LF\). Following Vign\'eras \cite{Vigneras:l-modulaires}*{Section I.3} we replace \(\Hecke(\G_\LF,\Q)\) by a Hecke algebra with \(\Z[\nicefrac{1}{p}]\)-coefficients.  This allows us to extend the correspondence between representations of~\(\G_\LF\) and \(\Hecke(\G_\LF)\)\nb-modules to representations on \(\Z[\nicefrac1p]\)-modules, thus covering vector spaces over fields of characteristic different from~\(p\).  Besides the non-degenerate \(\Hecke\)\nb-modules, which we call smooth here, we also need a dual class of rough \(\Hecke\)\nb-modules, which we introduce here (see also~\cite{Meyer:Smooth}).

\begin{lemma}
  \label{lem:subgroup_power_p}
  There is a compact open subgroup \(\CG\subseteq \G_\LF\) that is a pro-\(p\)-group, that is,  the index \([\CG:\CG']\) is a power of~\(p\) for all compact open subgroups~\(\CG'\) of~\(\CG\).
\end{lemma}

\begin{proof}
  Closed subgroups of pro-\(p\)-groups are again pro-\(p\)-groups.  Since any linear algebraic group is contained in~\(\Gl_d\) by definition, it suffices to prove the assertion for \(\Gl_d(\LF)\).  The subgroups \(\CG_n \defeq 1 + \Mat_d(\maxid^n)\) for \(n\ge1\) form a decreasing sequence of compact open subgroups and a neighbourhood basis of~\(1\).  Since \([\CG_n:\CG_{n+1}]=q^{d^2}\) is a power of~\(p\) for each~\(n\) and any open subgroup of~\(\CG_1\) contains~\(\CG_n\) for some \(n\in\N\), \(\CG_1\) is a pro-\(p\)-group.
\end{proof}

The lemma allows us to choose a Haar measure~\(\mu\) on~\(\G_\LF\) with \(\mu(\CG)\in\Z[\nicefrac{1}{p}]\) for all compact open subgroups~\(\CG\).  Let~\(\Hecke\) or, more precisely, \(\Hecke\bigl(\G_\LF,\Z[\nicefrac{1}{p}]\bigr)\) be the \(\Z[\nicefrac{1}{p}]\)-module of all locally constant, compactly supported functions \(\G_\LF\to \Z[\nicefrac{1}{p}]\).  Define the convolution of \(f_1,f_2\in\Hecke\) by
\[
f_1*f_2(g) \defeq \int_{\G_\LF} f_1(h) f_2(h^{-1}g) \,\diff\mu(h)
\qquad\text{for \(g\in\G_\LF\).}
\]
We claim that this belongs to~\(\Hecke\) again.  To see this, choose a compact open subgroup~\(\CG\) that is so small that~\(f_2\) is left \(\CG\)\nb-invariant and~\(f_1\) is right \(\CG\)\nb-invariant; then~\(f_1\) is a \(\Z[\nicefrac{1}{p}]\)-linear combinations of characteristic functions of cosets~\(g\CG\) for \(g\in \G_\LF\), and \(\chi_{g\CG}*f_2 = \mu(\CG)\cdot \lambda_gf_2\), where \(\lambda_gf_2(x) = f_2(g^{-1}x)\) as usual.

Thus~\(\Hecke\) is a ring over \(\Z[\nicefrac1p]\).  It is a subring of the \(\Q\)\nb-valued Hecke algebra \(\Hecke(\G_\LF,\Q) \defeq \Hecke\otimes_\Z \Q\).  The group~\(\G_\LF\) is embedded in the multiplier algebra of~\(\Hecke\), that is, products of the form \(gf\) or \(fg\) with \(g\in\G_\LF\) and \(f\in\Hecke\) are well-defined and satisfy the expected properties.

For a compact open pro-\(p\)-subgroup \(\CG\subseteq\G_\LF\), let
\[
\av{\CG}\defeq \mu(\CG)^{-1}\chi_\CG.
\]
This is an idempotent element in the ring~\(\Hecke\).  Let \((\CG_n)_{n\in\N}\) be a decreasing sequence of compact open pro-\(p\)-subgroups of~\(\G_\LF\) with \(\bigcap \CG_n=\{1\}\).  Then \(\bigl(\av{\CG_n}\bigr)_{n\in\N}\) is an increasing approximate unit of projections in~\(\G_\LF\).

\begin{definition}
  \label{def:smooth_module}
  An \(\Hecke\)\nb-module~\(V\) is called \emph{smooth} if \(V\cong \varinjlim {}\av{\CG_n}V\), that is, for each \(v\in V\) there is \(n\in\N\) with \(\av{\CG_n}\cdot v=v\).

  It is called \emph{rough} if \(V\cong \varprojlim {}\av{\CG_n}V\), that is, for any \((v_n)_{n\in\N} \in \prod_{n\in\N} V\) with \(\av{\CG_n}v_{n+1}=v_n\) for all \(n\in\N\), there is a unique \(v\in V\) with \(v_n = \av{\CG_n}\cdot v\) for all \(n\in\N\).

  We define the \emph{smoothening} \(\Smooth(V)\) and the \emph{roughening} \(\Rough(V)\) of an \(\Hecke\)\nb-module~\(V\) by
  \[
  \Smooth(V) \defeq \varinjlim {}\av{\CG_n}V,\qquad
  \Rough(V) \defeq \varprojlim {}\av{\CG_n}V,
  \]
  using the embedding \(\av{\CG_n}V\to \av{\CG_{n+1}}V\) and the projection \(\av{\CG_{n+1}}V\to \av{\CG_n}V\) induced by~\(\av{\CG_n}\) as structure maps.
\end{definition}

Since \(\av{\CG_n}V\) is a unital \(\av{\CG_n}\Hecke\av{\CG_n}\)-module and \(\Hecke= \varinjlim {}\av{\CG_n}\Hecke\av{\CG_n}\), both \(\Smooth(V)\) and \(\Rough(V)\) are modules over~\(\Hecke\).  Even more, the multiplier algebra \(\Mult(\Hecke)\) of~\(\Hecke\) acts on \(\Smooth(V)\) and \(\Rough(V)\) because for any multiplier~\(\mu\) of~\(\Hecke\), both \(\mu\av{\CG_n}\) and \(\av{\CG_n}\mu\) belong to \(\av{\CG_m}\Hecke\av{\CG_m}\) for some \(m\in\N\).  This allows us to well-define \(\mu\av{\CG_n}v\in \av{\CG_m}V\) for \(v\in V\) and \(\av{\CG_m}\mu v\in \av{\CG_m}V\) for \(v\in \Rough(V)\).  The canonical maps \(\Smooth(V)\to V\to\Rough(V)\) are \(\Mult(\Hecke)\)\nb-module homomorphisms.  In particular, since \(\G_\LF\subseteq\Mult(\Hecke)\), smooth and rough \(\Hecke\)\nb-modules both carry natural representations of~\(\G_\LF\).

We may alternatively define smoothenings and roughenings as \(\Smooth(V) \cong \Hecke \otimes_\Hecke V\) and \(\Rough(V) \cong \Hom_\Hecke(\Hecke,V)\).  These definitions are used in~\cite{Meyer:Smooth} and can be extended to all locally compact groups.

\begin{proposition}
  \label{pro:Hecke_modules}
  The category of smooth \(\Hecke\)\nb-modules is equivalent to the category of smooth representations of~\(\G_\LF\) on \(\Z[\nicefrac{1}{p}]\)-modules.

  Let \(V\) and~\(W\) be two \(\Hecke\)\nb-modules.  If~\(V\) is smooth, then the map \(\Smooth(W) \to W\) induces an isomorphism \(\Hom_\Hecke\bigl(V,\Smooth(W)\bigr) \cong \Hom_\Hecke(V,W)\).  If~\(W\) is rough, then the map \(V\to\Rough(V)\) induces an isomorphism \(\Hom_\Hecke(\Rough(V),W) \cong \Hom_\Hecke(V,W)\).

  Let~\(V\) be an \(\Hecke\)\nb-module.  The natural maps \(\Smooth(V)\to V\to\Rough(V)\) induce natural isomorphisms
  \[
  \Smooth\bigl(\Smooth(V)\bigr) \cong \Smooth(V) \cong \Smooth\bigl(\Rough(V)\bigr),\qquad
  \Rough\bigl(\Smooth(V)\bigr) \cong \Rough(V) \cong \Rough\bigl(\Rough(V)\bigr).
  \]

  The smoothening and roughening functors restrict to equivalences of categories between the subcategories of rough and smooth \(\Hecke\)\nb-modules, respectively.
\end{proposition}

\begin{proof}
  The first statement is well-known for \(\Hecke(\G_\LF,\Q)\), and the proof carries over literally to the \(\Z[\nicefrac1p]\)-linear case.

  Any \(\Hecke\)\nb-module homomorphism \(f\colon V\to W\) maps \(\av{\CG_n}V\) to \(\av{\CG_n}W\).  If~\(V\) is smooth, then \(V=\varinjlim {}\av{\CG_n}V\), so that~\(f\) factors through \(\varinjlim {}\av{\CG_n}W = \Smooth(W)\).  Thus \(\Hom_\Hecke\bigl(V,\Smooth(W)\bigr) \cong \Hom_\Hecke(V,W)\).  We also get induced maps
  \[
  \Rough(V) = \varprojlim_n {}\av{\CG_n}V \to \av{\CG_n}V\to \av{\CG_n}W
  \]
  for all~\(n\).  These piece together to a map \(\Rough(V)\to\Rough(W)\).  If~\(W\) is rough, this shows that~\(f\) extends uniquely to a map \(\Rough(V)\to W\), so that \(\Hom_\Hecke(\Rough(V),W) \cong \Hom_\Hecke(V,W)\).  The assertions in the third paragraph follow because \(\av{\CG_n}\Smooth(V) = \av{\CG_n}V = \av{\CG_n}\Rough(V)\) for all \(n\in\N\).  They show that \(\Smooth\) and~\(\Rough\) are inverse to each other as functors between the subcategories of rough and smooth representations, respectively, whence the equivalence of categories.
\end{proof}

Recall that both smooth and rough \(\Hecke(\G_\LF)\)-modules carry an induced group representation of~\(\G_\LF\).  Conversely, this representation of~\(\G_\LF\) determines the module structure, by integration.  Thus we may also speak of smooth and rough group representations of~\(\G_\LF\).  A representation is rough if and only if it is the projective limit of the subspaces of \(\CG_n\)\nb-invariants with respect to the averaging maps.

\subsubsection{Cellular chain complexes of equivariant cosheaves}
\label{sec:cellular_coefficient}

For any reductive \(p\)\nb-adic group, Bruhat and Tits \cites{Bruhat-Tits:Reductifs_I, Bruhat-Tits:Reductifs_II, Tits:Reductive} constructed an affine building.  More precisely, they constructed two buildings, one for~\(\G_\LF\) and one for its maximal semisimple quotient~\(\G^\ses_\LF\).  We shall use the building for~\(\G^\ses_\LF\), which we call the \emph{Bruhat--Tits building} of~\(\G_\LF\) and denote by \(\Buil\).

Recall that \(\Buil\) is a locally finite polysimplicial complex of dimension equal to the rank of~\(\G^\ses\).  It carries a canonical metric, for which it becomes a CAT(0)-space.  The group~\(\G^\ses_\LF\) acts on~\(\Buil\), properly, cocompactly and isometrically.  Being a CAT(0)-space, it follows that~\(\Buil\) is \(\CG\)\nb-equivariantly contractible for any compact subgroup~\(\CG\) of~\(\G^\ses_\LF\), so that~\(\Buil\) is a classifying space for proper actions of~\(\G^\ses_\LF\)  (see~\cite{Baum-Connes-Higson:BC}).  The action of~\(\G^\ses_\LF\) induces one of~\(\G_\LF\) because of the quotient map \(\G\epi\G^\ses\).

We mostly treat polysimplicial complexes such as~\(\Buil\) as purely combinatorial objects and view~\(\Buil\) as the set of polysimplices, partially ordered by \(\tau\prec\sigma\) if~\(\tau\) is a face of~\(\sigma\).  (Hence it would make no big difference if we used the building for~\(\G_\LF\) instead of the building for~\(\G_\LF^\ses\).)  A polysimplex of dimension~\(0\) is called a \emph{vertex}, and a polysimplex of maximal dimension is called a \emph{chamber}.  For a polysimplicial complex~\(\Sigma\), we let \(\Sigma^\circ\) be its set of vertices.  Two vertices or polysimplices \(x\) and~\(y\) are called \emph{adjacent} if there is a polysimplex~\(\sigma\) with \(x,y\prec\sigma\); adjacent vertices need not be connected by an edge unless~\(\Sigma\) is a simplicial complex.  The \emph{star} of a polysimplex is the set of all polysimplices adjacent to it.  If \(\sigma\) and~\(\tau\) are adjacent, then we let \([\sigma,\tau]\) be the smallest polysimplex containing \(\sigma\cup\tau\).

The action of~\(\G_\LF\) on~\(\Buil\) preserves the polysimplicial structure, so that we get an induced action on the set of polysimplices.

Now we recall how to construct chain and cochain complexes of representations using (simplicial) cosheaves and sheaves on~\(\Buil\) (see also~\cite{Schneider-Stuhler:Rep_sheaves}*{Section II.1}).  Cosheaves are also called coefficient systems.

Let~\(\Sigma\) be a polysimplicial complex.  A \emph{sheaf} on~\(\Sigma\) is a system of Abelian groups~\((V_\sigma)_{\sigma\in\Sigma}\) with maps \(\varphi_\sigma^\tau\colon V_\sigma\to V_\tau\) for \(\tau\succ\sigma\) that satisfy \(\varphi_\sigma^\sigma=\Id_{V_\sigma}\) and \(\varphi_\tau^\omega\circ \varphi_\sigma^\tau=\varphi_\sigma^\omega\) for \(\omega\succ\tau\succ\sigma\).  In other words, a sheaf is a functor on the category associated to the partially ordered set \((\Sigma,\prec)\).  Dually, a \emph{cosheaf} on~\(\Sigma\) is contravariant functor on this category, that is, a system of Abelian groups~\((V_\sigma)_{\sigma\in\Sigma}\) with maps \(\varphi_\sigma^\tau\colon V_\sigma\to V_\tau\) for \(\tau\prec\sigma\) that satisfy \(\varphi_\sigma^\sigma=\Id_{V_\sigma}\) and \(\varphi_\tau^\omega\circ \varphi_\sigma^\tau=\varphi_\sigma^\omega\) for \(\omega\prec\tau\prec\sigma\).

To form cellular chain complexes, we equip each simplex with an orientation.  This induces orientations on its boundary faces.  We define
\[
\varepsilon_{\tau\sigma} \defeq
\begin{cases}
  \phantom{-}1 &\text{if \(\tau\prec\sigma\) with compatible orientations,}\\
  -1 &\text{if \(\tau\prec\sigma\) with opposite orientations,}\\
  \phantom{-}0 &\text{if~\(\tau\) is not a face of~\(\sigma\).}
\end{cases}
\]
Let \(\Gamma =(V_\sigma,\varphi_\sigma^\tau)\) be a cosheaf on a polysimplicial complex~\(\Sigma\).  The \emph{cellular chain complex} \(\Cell\bigl(\Sigma,\Gamma\bigr)\) of~\(\Sigma\) with coefficients~\(\Gamma\) is the \(\N\)\nb-graded chain complex \(\bigoplus_{\sigma\in\Sigma} V_\sigma\) with~\(V_\sigma\) in degree \(\deg(\sigma)\) and with the boundary map
\[
\partial\bigl((v_\sigma)_{\sigma\in\Sigma}\bigr)_\tau \defeq \sum_{\sigma\in\Sigma} \varepsilon_{\tau\sigma} \varphi_\sigma^\tau(v_\sigma).
\]
The homology of \(\Cell(\Sigma,\Gamma)\) is denoted by \(\Ho_*(\Sigma,\Gamma)\) and called the \emph{homology of~\(\Sigma\) with coefficients~\(\Gamma\)}.

Dually, let \(\Gamma =(V_\sigma,\varphi_\sigma^\tau)\) be a sheaf on~\(\Sigma\) and assume that~\(\Sigma\) is locally finite -- this holds for subcomplexes of \(\Buil \).  The \emph{cellular cochain complex} \(\DCell\bigl(\Sigma,\Gamma\bigr)\) of~\(\Sigma\) with coefficients~\(\Gamma\) is the \(\N\)\nb-graded cochain complex \(\prod_{\sigma\in\Sigma} V_\sigma\) with~\(V_\sigma\) in degree \(\deg(\sigma)\) and with the boundary map
\[
\partial\bigl((v_\sigma)_{\sigma\in\Sigma}\bigr)_\tau \defeq \sum_{\sigma\in\Sigma} \varepsilon_{\sigma\tau} \varphi_\sigma^\tau(v_\sigma),
\]
which is well-defined because~\(\Sigma\) is locally finite.  The cohomology of \(\DCell(\Sigma,\Gamma)\) is denoted by \(\Ho^*(\Sigma,\Gamma)\) and called the \emph{cohomology of~\(\Sigma\) with coefficients~\(\Gamma\)}.

A \emph{\(\G_\LF\)\nb-equivariant cosheaf or sheaf} on~\(\Buil\) is a cosheaf or sheaf~\(\Gamma\) on~\(\Buil\) with isomorphisms \(\alpha_g\colon V_\sigma \congto V_{g\cdot\sigma}\) for all \(g\in\G_\LF\), \(\sigma\in\Buil\) compatible with the maps~\(\varphi_\sigma^\tau\), such that \(\alpha_1=\Id\), \(\alpha_g\circ\alpha_h = \alpha_{gh}\).  The cellular (co)chain complex of a \(\G_\LF\)\nb-equivariant (co)sheaf inherits a representation of~\(\G_\LF\) by
\[
\alpha_g\bigl((v_\sigma)_{\sigma\in\Sigma}\bigr)_\tau \defeq
\sum_{\sigma\in\Sigma} g_{\tau\sigma} \alpha_g(v_\sigma),
\]
where
\begin{equation}
  \label{eq:orientation_character}
  g_{\tau\sigma}=
  \begin{cases}
    1&\text{if \(g(\sigma)=\tau\) and \(g|_\sigma\colon \sigma\to\tau\) preserves orientations,}\\
    -1&\text{if \(g(\sigma)=\tau\) and \(g|_\sigma\colon \sigma\to\tau\) reverses orientations,}\\
    0&\text{otherwise.}
  \end{cases}
\end{equation}
Each~\(V_\sigma\) inherits a representation of the stabiliser
\[
\StabS_\sigma \defeq \{g\in\G_\LF\mid g\sigma=\sigma\}.
\]
Notice that this group may be strictly larger than the pointwise stabiliser
\[
\StabP_\sigma \defeq \{g\in\G_\LF\mid \text{\(gx=x\) for each vertex~\(x\) of~\(\sigma\)}\}.
\]

\begin{lemma}
  \label{lem:cochain_smooth_chain_rough}
  The representation of~\(\G_\LF\) on~\(\Cell(\Buil,\Gamma)\) is smooth if and only if~\(\StabP_\sigma\) acts smoothly on~\(V_\sigma\) for each polysimplex~\(\sigma\) in~\(\Buil\).

  The representation of~\(\G_\LF\) on~\(\DCell(\Buil,\hat{\Gamma})\) is rough if and only if~\(\StabP_\sigma\) acts roughly on~\(V_\sigma\) for each polysimplex~\(\sigma\) in~\(\Buil\).
\end{lemma}

\begin{proof}
  We may replace~\(\StabP_\sigma\) by~\(\StabS_\sigma\) in both statements because the former is an open subgroup of~\(\StabS_\sigma\).

  Let~\(S\) be a set of representatives for the orbits of~\(\G_\LF\) on \(\Buil\).  This set is finite because~\(\G_\LF\) acts transitively on the set of chambers.  As a representation of~\(\G_\LF\) 
  \[
  \Cell(\Buil,\Gamma) = \bigoplus_{\sigma\in S} \cInd_{\StabS_\sigma}^{\G_\LF} V_\sigma,
  \]
  where we equip~\(V_\sigma\) with the induced representation of~\(\StabS_\sigma\), twisted by the orientation character in~\eqref{eq:orientation_character}, and where \(\cInd_{\StabS_\sigma}^{\G_\LF} V_\sigma\) is the space of all functions \(f\colon \G_\LF\to V_\sigma\) with \(g f(xg)= f(x)\) for \(g\in\StabS_\sigma\) and \(f(x)=0\) for~\(x\) outside a compact subset of \(\G_\LF/\StabS_\sigma\).  The group~\(\G_\LF\) acts on this by left translation.  It is easy to see that this representation is smooth if~\(\StabS_\sigma\) acts smoothly on~\(V_\sigma\).
  Similarly,
  \[
  \DCell(\Buil,\Gamma) = \prod_{\sigma\in S} \Ind_{\StabS_\sigma}^{\G_\LF} V_\sigma,
  \]
  where~\(V_\sigma\) carries the same representation as above and \(\Ind_{\StabS_\sigma}^{\G_\LF} V_\sigma\) is defined like \(\cInd_{\StabS_\sigma}^{\G_\LF} V_\sigma\) but without the support restriction.  Such a representation of~\(\G_\LF\) is usually not smooth, even if~\(\StabS_\sigma\) acts smoothly on~\(V_\sigma\), because there is no uniformity in the smoothness of functions in \(\Ind_{\StabS_\sigma}^{\G_\LF} V_\sigma\).  

Let \((X_n)\) be an increasing sequence of \(\StabS_\sigma\)\nb-biinvariant subsets of~\(\G_\LF\) with \(\G_\LF=\bigcup X_n\).  By definition, \(\Ind_{\StabS_\sigma}^{\G_\LF} V_\sigma\) is the \emph{projective} limit of the spaces of functions in \(\Ind_{\StabS_\sigma}^{\G_\LF} V_\sigma\) that are supported in~\(X_n\).  The group~\(\StabS_\sigma\) acts smoothly or roughly on this subspace if and only if it acts smoothly or roughly on~\(V_\sigma\).  The induced representation of~\(\StabS_\sigma\) on the projective limit of these rough representations remains rough.  This is equivalent to roughness as a representation of~\(\G_\LF\) because~\(\StabS_\sigma\) is open in~\(\G_\LF\).
\end{proof}

\begin{definition}
  \label{def:hull}
  We define the hull \(\Hull (\sigma,\tau)\) of two polysimplices \(\sigma\) and~\(\tau\) in a building as the intersection of all apartments containing \(\sigma\cup\tau\) (see Figure~\ref{fig:hull} for some examples).
\end{definition}
\begin{figure}[htbp]
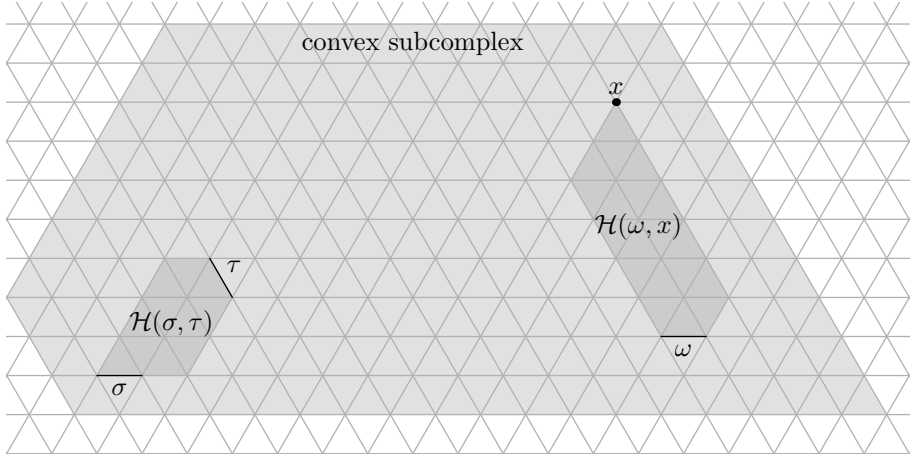

  \centering
\begin{asy}
unitsize(.3cm);
fill(1*a[1]+a[0]--2*a[0]+4*a[2]--9*a[0]+11*a[2]--20*a[0]+11*a[2]--20*a[0]+a[2]--cycle,white*.88);//big convex subcomplex

fill((2*a[1]+a[0])--(2*a[1]+3*a[0])--(4a[1]+3a[0])--(5*a[1]+2*a[0])--(5*a[1]+a[0])--cycle,white*.80);//first convex hull
fill((7*a[1]+9*a[0])--(9*a[1]+9*a[0])--(4*a[1]+14*a[0])--(3*a[1]+14*a[0])--(3*a[1]+13*a[0])--cycle,white*.80);//second convex hull

//draw the walls of the building:
for(i=-18;i<30;++i){
  draw((i*a[0])--(i*a[0]+20*a[1]),white*.7);
  draw((i*a[0])--(i*a[0]+20*a[2]),white*.7);
}
for(i=0;i<20;++i){
  draw((i*a[2]-10*a[0])--(i*a[2]+(20+i)*a[0]),white*.7);
}

//clip to rectangular region:
clip((0,0)--(40,0)--(40,20)--(0,20)--cycle);

//facets defining the first convex hull:
draw ((2*a[1]+a[0])--(2*a[1]+2*a[0]));
label("\(\sigma\)",2*a[1]+1.5*a[0],S);
draw ((4*a[1]+3*a[0])--(4*a[1]+3*a[0]+a[2]));
label("\(\tau\)",4*a[1]+3*a[0]+.5*a[2],NE);
label ("\(\Hull(\sigma,\tau)\)",1.9*a[0]+3.5*a[1]);

//facets defining the second convex hull:
dot("\(x\)",(9*a[1]+9*a[0]),N);
draw((3*a[1]+13*a[0])--(3*a[1]+14*a[0]));
label("\(\omega\)",3*a[1]+13.5*a[0],S);
label ("\(\Hull(\omega,x)\)",6*a[1]+11*a[0]);

//label convex subcomplex
label ("convex subcomplex",11*a[2]+14.5*a[0],S);
\end{asy}
  \caption{Hulls of facets and a convex subcomplex in an \(\tilde{A}_2\)\nb-apartment}
  \label{fig:hull}
\end{figure}
This notion generalizes the hull of two chambers, as defined in \cite{Garrett:Buildings}*{\S16.2}.

\begin{definition}
  \label{def:subcomplex}
  A \emph{subcomplex} of a polysimplicial complex~\(\Sigma\) is a subset~\(\Sigma'\) of~\(\Sigma\) with \(\tau\in\Sigma'\) if \(\tau\prec\sigma\) and \(\sigma\in\Sigma'\).
\end{definition}

\begin{definition}
  \label{def:convex_subcomplex}
  A subcomplex~\(\Sigma\) of~\(\Buil\) is called \emph{convex} if any polysimplex contained in \(\Hull(\sigma,\tau)\) for \(\sigma,\tau\in\Sigma\), is contained in~\(\Sigma\).
\end{definition}

By definition, \(\Hull(\sigma,\tau)\) is the smallest convex subcomplex containing \(\sigma\cup\tau\).  A subcomplex~\(\Sigma\) of~\(\Buil\) is convex in this combinatorial sense if and only if its geometric realisation~\(\realise{\Sigma}\) is convex in \(\realise{\Buil}\) in the geometric sense: \(x\in\realise{\Sigma}\) if~\(x\) lies on the geodesic segment between two points of~\(\realise{\Sigma}\).

\begin{example}
  \label{exa:convex}
  The star of a polysimplex in~\(\Buil\) is a convex subcomplex.

  Let \(\CG\subseteq\G_\LF\) be a compact subgroup.  Define \(\Buil^\CG\subseteq\Buil\) by \(\sigma\in\Buil^\CG\) if and only if all vertices of~\(\sigma\) are fixed by~\(\CG\).  This is a non-empty convex subcomplex of~\(\Buil\).  It is finite if the subgroup~\(\CG\) is compact and open.
\end{example}

\section{Natural resolutions of representations}
\label{sec:resolve_rep}

Peter Schneider and Ulrich Stuhler~\cite{Schneider-Stuhler:Rep_sheaves} associated a certain cosheaf to an admissible \(\Q\)\nb-linear representation~\(V\) of a reductive \(p\)\nb-adic group~\(\G_\LF\) and showed that the cellular chain complex with coefficients in this cosheaf is a resolution of~\(V\).  Their proof was indirect and based on a deep result of Joseph Bernstein about Serre subcategories of the category of smooth representations (\cite{Bernstein:Centre}*{Corollaire 3.9}).  Marie-France Vign\'eras~\cite{Vigneras:Cohomology} extended the constructions in~\cite{Schneider-Stuhler:Rep_sheaves} to representations over fields of characteristic different from~\(p\), based on the results of~\cite{Vigneras:l-modulaires}.

We are going to prove directly that cellular (co)chain complexes with certain (co)sheaves as coefficients are acyclic and compute their (co)homology in degree~\(0\).  It is important for the proof and for some applications, such as the character computations below, to allow finite convex subcomplexes of the building.

The cosheaves considered in \cites{Schneider-Stuhler:Rep_sheaves, Vigneras:Cohomology} are of the form \(V_\sigma \defeq V^{\CG_\sigma}\) for certain compact open subgroups \(\CG_\sigma\subseteq\G_\LF\); here \(V^{\CG_\sigma}\subseteq V\) denotes the subspace of \(\CG_\sigma\)\nb-invariants.  Since the subgroup~\(\CG_\sigma\) can be computed from the groups~\(\CG_x\) for the vertices of~\(\sigma\), it suffices to describe the subgroups~\(\CG_x\) for vertices \(x\in\Buil^\circ\).

The subgroups used in~\cite{Schneider-Stuhler:Rep_sheaves} are small enough to be pro-\(p\)-groups (see Lemma~\ref{lem:subgroup_power_p}) and hence give rise to idempotents \(\av{\CG_\sigma}\in\Hecke = \Hecke\bigl(\G_\LF,\Z[\nicefrac{1}{p}]\bigr)\).  The subspace~\(V^{\CG_\sigma}\) is the range \(\av{\CG_\sigma}V\) of this idempotent on~\(V\).  These idempotents are more relevant than the subgroups for our proofs, which therefore break down in characteristic~\(p\).  In the next section, we formalise the required properties of the idempotents~\(\av{\CG_x}\).

\subsection{Consistent systems of idempotents}
\label{sec:consistent_idempotents}

\begin{definition}
  \label{def:consistent_idempotents}
  A system \((\Proj_x)_{x\in\Buil^\circ}\) of idempotent endomorphisms \(\Proj_x\colon V\to V\) is called \emph{consistent} if it has the following properties:
  \begin{enumerate}[label=\textup{(\alph{*})}]
  \item \label{consistent_commute} \(\Proj_x\) and~\(\Proj_y\) commute if \(x\) and~\(y\) are adjacent;
  \item \label{consistent_convex} \(\Proj_x\Proj_z\Proj_y=\Proj_x\Proj_y\) for \(x,y,z\in\Buil^\circ\) with \(z\in\Hull(x,y)\) and~\(z\) is adjacent to~\(x\).
  \end{enumerate}
  If \(\pi\colon \G_\LF\to\Aut(V)\) is a group representation, then a system of idempotents is called \emph{equivariant} if
  \begin{enumerate}[resume,label=\textup{(\alph{*})}]
  \item \label{consistent_equivariant} \(\Proj_{gx} = \pi_g\Proj_x\pi_g^{-1}\) for all \(g\in\G_\LF\), \(x\in\Buil^\circ\).
  \end{enumerate}
\end{definition}

The idempotents~\(\Proj_x\) for vertices~\(x\) yield idempotents~\(\Proj_\sigma\) for polysimplices~\(\sigma\), which inherit analogues of the consistency properties:

\begin{proposition}
  \label{pro:consistent}
  Let \((\Proj_x)_{x\in\Buil^\circ}\) be a consistent system of idempotents.  For a polysimplex \(\sigma\in\Buil^\circ\),
  \[
  \Proj_\sigma \defeq \prod_{\twolinesubscript{x\in\Buil^\circ}{x \prec\sigma}} \Proj_x
  \]
  is a well-defined idempotent endomorphism of~\(V\).
  \begin{enumerate}[resume,label=\textup{(\alph{*})}]
  \item \label{consistent_simplex_commute} \(\Proj_\sigma\Proj_\tau=\Proj_{[\sigma,\tau]}\) if the polysimplices \(\tau\) and~\(\sigma\) are adjacent; here \([\sigma,\tau]\) denotes the smallest polysimplex containing \(\sigma\) and~\(\tau\);

  \item \label{consistent_simplex_convex} \(\Proj_\sigma\Proj_\omega\Proj_\tau=\Proj_\sigma\Proj_\tau\) if \(\sigma\), \(\tau\), and~\(\omega\) are polysimplices in~\(\Buil\) with \(\omega\in \Hull(\sigma,\tau)\);
  \end{enumerate}
  if \((\Proj_x)_{x\in\Buil^\circ}\) is equivariant, then
  \begin{enumerate}[resume,label=\textup{(\alph{*})}]
  \item \label{consistent_simplex_equivariant} \(\pi_g \Proj_\sigma \pi_g^{-1} = \Proj_{g\cdot\sigma}\) for all \(g\in\G_\LF\), \(\sigma\in\Buil\).
  \end{enumerate}
\end{proposition}

Proposition~\ref{pro:consistent} is interesting from an axiomatic point of view although our (co)homology computations only require a small part of it and checking \ref{consistent_simplex_commute}--\ref{consistent_simplex_equivariant} is not much harder in examples than checking \ref{consistent_commute}--\ref{consistent_equivariant}.  Since the proof of Proposition~\ref{pro:consistent} is rather complicated, we postpone it to Section~\ref{sec:proof_consistent}.

Given a consistent system of idempotents, we define
\[
V_\sigma \defeq \Proj_\sigma(V)\subseteq V
\]
and let \(\varphi_\sigma^\tau\colon V_\sigma\to V_\tau\) for \(\tau\prec\sigma\) be the inclusion map (here we use~\ref{consistent_simplex_commute}).  This defines a cosheaf on~\(\Buil\), which we denote by~\(\Gamma\).  If the system \((\Proj_x)\) is equivariant, \(\Gamma\) is a \(\G_\LF\)\nb-equivariant cosheaf by~\ref{consistent_simplex_equivariant}.

The cellular chain complex \(\Cell(\Buil,\Gamma)\) is augmented by the map
\[
\alpha\colon \Cell[0](\Buil,\Gamma) = \bigoplus_{x\in\Buil^\circ} V_x \to V,
\]
taking the embedding \(V_x\mono V\) on the summand~\(V_x\).  Clearly, \(\alpha\) is \(\G_\LF\)\nb-equivariant and satisfies \(\alpha\circ\partial=0\).

We also let \(\hat\varphi_\tau^\sigma\colon V_\tau\to V_\sigma\) for \(\tau\prec\sigma\) be the projection~\(\Proj_\sigma\); this is a split surjection by~\ref{consistent_simplex_commute}.  This defines a sheaf \(\hat{\Gamma} = (V_\sigma,\hat\varphi_\tau^\sigma)\) on~\(\Buil\).  It is \(\G_\LF\)\nb-equivariant if~\((\Proj_x)\) is equivariant.  Its cellular chain complex is augmented by the equivariant chain map
\[
\alpha\colon V \to  \DCell[0](\Buil,\hat{\Gamma}) = \prod_{x\in\Buil^\circ} V_x,\qquad
v\mapsto \bigl(\Proj_x(v)\bigr)_{x\in\Buil^\circ}.
\]
Condition~\ref{consistent_convex} is not necessary for \(\Gamma\) and \(\hat{\Gamma}\) to be equivariant simplicial (co)sheaves, but to prove acyclicity of \(\Cell(\Buil,\Gamma)\) and \(\DCell(\Buil,\hat{\Gamma})\).

Although the sheaf and cosheaf \(\Gamma\) and \(\hat{\Gamma}\) seem unrelated at first sight, these two constructions become equivalent when we allow~\(V\) to be an object of a general Abelian category~\(\Abcat\).

In this setting, the collection of endomorphisms of~\(V\) is still a ring, so that idempotents in \(\Endo(V)\) make sense. A consistent system of idempotents in \(\Endo(V)\) for an object~\(V\) of an Abelian category~\(\Abcat\) yields a cosheaf \(\Gamma\) and a sheaf \(\hat{\Gamma}\) with values in~\(\Abcat\) exactly as above.  We may form the cellular chain and cochain complexes \(\Cell(\Sigma,\Gamma)\) and \(\DCell(\Sigma,\hat{\Gamma})\) provided~\(\Abcat\) has countable coproducts and products.  For (co)homology computations, we require these coproducts and products to be exact.

\begin{lemma}
  \label{lem:opposite_category}
  The passage from~\(\Abcat\) to its opposite category~\(\Abcat^\op\) exchanges the roles of \(\Gamma\) and \(\hat{\Gamma}\) and hence of \(\Cell(\Sigma,\Gamma)\) and \(\DCell(\Sigma,\hat{\Gamma})\).
\end{lemma}

This is why it is useful to allow general categories in the following, although we are mainly interested in representations on \(\Z[\nicefrac1p]\)-modules or on vector spaces over some field.

\begin{proof}
  Since \(\Endo_{\Abcat^\op}(V)\) is the opposite ring of \(\Endo_\Abcat(V)\), both rings \(\Endo_{\Abcat^\op}(V)\) and \(\Endo_\Abcat(V)\) have the same idempotents.  Thus the constructions in \(\Abcat\) and~\(\Abcat^\op\) use the same data.  Conditions \ref{consistent_simplex_commute}--\ref{consistent_simplex_equivariant} in Proposition~\ref{pro:consistent} are manifestly invariant under passage to the opposite ring, so that we get the same consistent or equivariant systems of idempotents in \(\Abcat\) and~\(\Abcat^\op\).  Now consider an idempotent endomorphism~\(p\) of~\(V\) as an endomorphism in~\(\Abcat^\op\).  Its range remains \(p(V)\), and the embedding \(p(V)\to V\) becomes the quotient map \(V\to p(V)\) induced by~\(p\).  As a consequence, the construction of \(\Gamma\) in~\(\Abcat^\op\) yields precisely \(\hat{\Gamma}\).  Furthermore, the passage to opposite category exchanges products and coproducts, so that \(\Cell(\Sigma,\Gamma)\) becomes \(\DCell(\Sigma,\hat{\Gamma})\) in the opposite category, for any subcomplex~\(\Sigma\) of~\(\Buil\).
\end{proof}

\begin{theorem}
  \label{the:main_theorem_resolution}
  Let~\(\Abcat\) be an Abelian category with exact countable products and coproducts.  Let~\(V\) be an object of~\(\Abcat\) and let \((\Proj_x)_{x\in\Buil^\circ}\) be a consistent system of idempotents in its endomorphism ring \(\Endo(V)\).  Let~\(\G_\LF\) be a reductive \(p\)\nb-adic group and let~\(\Sigma\) be a convex subcomplex of its affine Bruhat--Tits building~\(\Buil\).  Let~\(I\) denote the directed set of finite convex subcomplexes of\/~\(\Sigma\).
  \begin{itemize}
  \item The cellular chain complex \(\Cell(\Sigma,\Gamma)\) is exact except in degree~\(0\), where the augmentation map induces an isomorphism
    \[
    \Ho_0(\Sigma,\Gamma)
    \cong \varinjlim_{\Sigma_\fin\in I} {}\sum_{x\in\Sigma_\fin^\circ} \Proj_x(V).
    \]
  \item The cellular cochain complex \(\DCell(\Sigma,\hat\Gamma)\) is exact except in degree~\(0\), where the augmentation map induces an isomorphism
    \[
    \varprojlim_{\Sigma_\fin\in I} {}\biggl(V \Bigm/ \bigcap_{x\in\Sigma_\fin^\circ} \ker \Proj_x \biggr) \cong
    \Ho^0(\Sigma,\hat\Gamma).
    \]
  \item If\/~\(\Sigma\) is itself finite, then the composite map
    \[
    \sum_{x\in\Sigma^\circ} \Proj_x(V) \cong
    \Ho_0(\Sigma,\Gamma) \to V \to
    \Ho^0(\Sigma,\hat\Gamma) \cong
    V \Bigm/ \bigcap_{x\in\Sigma^\circ} \ker \Proj_x
    \]
    is an isomorphism, that is,
    \[
    V \cong \sum_{x\in\Sigma^\circ} \Proj_x(V) \oplus \bigcap_{x\in\Sigma^\circ} \ker \Proj_x.
    \]
  \end{itemize}
\end{theorem}

Here we define \(\sum_{x\in\Sigma^\circ} \Proj_x(V)\) as the image of the map \(\bigoplus_{x\in\Sigma^\circ} \Proj_x(V)\to V\) and \(\bigcap_{x\in\Sigma^\circ} \ker(\Proj_x)\) as the infimum of \(\ker(\Proj_x)\) for \(x\in\Sigma^\circ\), which is the kernel of the map \(V\to \prod_{x\in\Sigma^\circ} \Proj_x(V)\).

\begin{remark}
  \label{rem:warn_projlim_intersect}
  Although \(\bigcap_{x\in\Sigma^\circ} \ker(\Proj_x) \cong \varprojlim_{\Sigma_\fin\in I} {} \bigcap_{x\in\Sigma_\fin^\circ} \ker(\Proj_x)\), we usually have
  \[
  \varprojlim_{\Sigma_\fin\in I} {}\biggl(V \Bigm/ \bigcap_{x\in\Sigma_\fin^\circ} \ker \Proj_x \biggr) \not\cong V \Bigm/ \bigcap_{x\in\Sigma^\circ} \ker(\Proj_x),
  \]
  already for irreducible smooth representations on \(\Q\)\nb-vector spaces.  The right hand side is a smooth representation in this case, while the cohomology of \(\DCell(\Sigma,\hat\Gamma)\) is a rough representation of~\(\G_\LF\) by Lemma~\ref{lem:cochain_smooth_chain_rough}.  But infinite-dimensional irreducible smooth representations are not rough (see also Proposition~\ref{pro:cohomology_roughening_Noetherean}).
\end{remark}

Theorem~\ref{the:main_theorem_resolution} is the main result of this article.  Its proof fills Section~\ref{sec:proof_resolution}.  The first assertion is the most important one and generalises results in \cites{Schneider-Stuhler:Rep_sheaves, Vigneras:Cohomology}.  The assertions about sheaf cohomology and its comparison with cosheaf homology appear to be new.  In our categorical formulation, they are equivalent to the corresponding statements about cosheaf homology.

\subsection{Some examples of consistent systems of idempotents}
\label{sec:examples_consistent}

Now we consider some special cases of Theorem~\ref{the:main_theorem_resolution}.  In these applications, \(\Abcat\) is a category of modules or vector spaces.  First we consider the case where \(\Proj_x=\av{\CG_x}\) for compact open subgroups \(\CG_x\subseteq\G_\LF\).

\begin{lemma}
  \label{lem:consistent_subgroups}
  Let \((\CG_x)_{x\in\Buil^\circ}\) be a system of compact open pro-\(p\)-subgroups of~\(\G_\LF\) and let~\(V\) be \(\Z[\nicefrac{1}{p}]\)-linear, so that the representation of~\(\G_\LF\) on~\(V\) integrates to a representation \(\Hecke\to\Endo(V)\).  Assume
  \begin{enumerate}[resume,label=\textup{(\alph{*})}]
  \item \label{consistent_group_commute} \(\CG_x\cdot \CG_y= \CG_y\cdot \CG_x\) if \(x\) and~\(y\) are adjacent;
  \item \label{consistent_group_convex} \(\CG_z\subseteq \CG_x\cdot \CG_y\) if \(x,y,z\in\Buil^\circ\), \(z\in\Hull(x,y)\), and~\(z\) is adjacent to~\(x\);
  \item \label{consistent_group_equivariant} \(g\CG_xg^{-1} = \CG_{gx}\) for all \(g\in\G_\LF\), \(x\in\Buil\).
  \end{enumerate}
  Then the system of idempotents \(\Proj_x\defeq \av{\CG_x}\) is consistent and equivariant,
  and
  \[
  \CG_\sigma = \prod_{\twolinesubscript{x\in\Buil^\circ}{x\prec\sigma}} \CG_x
  \]
  for a polysimplex~\(\sigma\) in~\(\Buil\) is a compact open subgroup of~\(\G_\LF\) with \(\Proj_\sigma=\av{\CG_\sigma}\).

  Conversely, \ref{consistent_group_commute}--\ref{consistent_group_equivariant} are necessary for \((\Proj_x)\) to be consistent as left multiplication operators on~\(\Hecke\).
\end{lemma}

Here we use the naive product of subsets
\[
A\cdot B \defeq \{a\cdot b\mid a\in A, b\in B\}
\qquad\text{for \(A,B\subseteq\G_\LF\).}
\]

\begin{proof}
  Since~\(V\) is a module over~\(\Hecke\) and the latter acts faithfully on itself, it suffices to show that the idempotents~\(\Proj_x\) satisfy Conditions \ref{consistent_commute}--\ref{consistent_equivariant} if and only if the subgroups~\(\CG_x\) satisfy \ref{consistent_group_commute}--\ref{consistent_group_equivariant}.

  Since \(\Proj_x\Proj_y\) and \(\Proj_y\Proj_x\) are supported on \(\CG_x\CG_y\) and \(\CG_y\CG_x\), 
respectively, \ref{consistent_group_commute} is necessary for Condition~\ref{consistent_commute} in Definition~\ref{def:consistent_idempotents}.  Conversely, if \(\CG_x\CG_y=\CG_y\CG_x\), then this is a compact open subgroup and \(\Proj_x\Proj_y=\av{\CG_x\CG_y}\).  The same argument yields the description of~\(\Proj_\sigma\) for a polysimplex~\(\sigma\).  The equivalence between \ref{consistent_equivariant} and~\ref{consistent_group_equivariant} is manifest.  Condition~\ref{consistent_group_convex} implies~\ref{consistent_convex} because \(\av{\CG_x}\cdot gh\cdot \av{\CG_y}= \av{\CG_x}\av{\CG_y}\) if \(g\in \CG_x\), \(h\in \CG_y\).  Conversely, the convolutions \(\av{\CG_x}\av{\CG_y}\) and \(\av{\CG_x}\av{\CG_z}\av{\CG_y}\) are supported in \(\CG_x\cdot \CG_y\) and \(\CG_x\cdot \CG_z\cdot \CG_y\), respectively.  Hence \(\Proj_x\Proj_z\Proj_y=\Proj_x\Proj_y\) implies \(\CG_x\cdot \CG_y = \CG_x\cdot \CG_z\cdot \CG_y \supseteq \CG_z\).
\end{proof}

Conditions \ref{consistent_group_commute} and~\ref{consistent_group_equivariant} are enough to get a cosheaf on the building.  We need~\ref{consistent_group_convex} to compute the homology of \(\Cell(\Buil,\Gamma)\).

It is rather easy to find systems of subgroups satisfying only \ref{consistent_group_commute} and~\ref{consistent_group_equivariant}.  Let \(\StabP_A \subseteq \G_\LF\) for a subcomplex~\(A\) of~\(\Buil\) denote the subgroup of all \(g\in\G_\LF\) that fix all simplices in~\(A\).  For each orbit in \(\G_\LF\backslash \Buil^\circ\), pick a representative \(x\in\Buil^\circ\) and a subgroup~\(\CG_x\) of~\(\StabP_{\Star x}\) that is normal in~\(\StabP_x\) (recall that the star of~\(x\) consists of all polysimplices in \(\Buil\) that are adjacent to~\(x\)); extend this to all of~\(\Buil^\circ\) by \(\CG_{gx} = g\CG_xg^{-1}\) for \(g\in\G_\LF\).  This makes sense because~\(\CG_x\) is normal in~\(\StabP_x\), and satisfies~\ref{consistent_group_equivariant} by construction.  If \(x\) and~\(y\) are adjacent, then \(\CG_y\subseteq\StabP_{\Star y}\subseteq \StabP_x\) normalises~\(\CG_x\), so that \(g\CG_x = \CG_xg\) for all \(g\in \CG_y\).  This yields~\ref{consistent_group_commute}.

The subgroups considered by Schneider and Stuhler satisfy~\ref{consistent_group_convex}; \cite{Vigneras:Cohomology}*{Lemma 1.28} checks this only for points on the straight line between \(x\) and~\(y\), but the same argument works if we merely assume \(z\in \Hull(x,y)\).

\smallskip

Now let~\(\G_\LF\) be the general linear group~\(\Gl_d(\LF)\) for some \(d\in\N \).  We denote its affine Bruhat--Tits building by~\(\BuilGl\).  A special feature of this group is that it acts transitively on the vertices of~\(\BuilGl\).  Hence an equivariant system of idempotents \((\Proj_x)_{x\in\BuilGl^\circ}\) is already specified by a single idempotent.

First we recall the structure of~\(\BuilGl\).  Let~\(\integ\) be the maximal compact subring of~\(\LF\) and let~\(\maxid\) be the maximal ideal in~\(\integ\).  Let~\(q\) be the cardinality of the residue field \(\integ/\maxid\).  Let \(\unif\in\maxid\) be a uniformiser, that is, \(\maxid = \unif\cdot\integ\).  We write \(\maxid^n\defeq \unif^n\cdot\integ\) for \(n\in\Z\).

A \emph{lattice} in~\(\LF^d\) is an \(\integ\)\nb-submodule of~\(\LF^d\) isomorphic to the \emph{standard lattice}~\(\integ^d\).  Two lattices \(\Lambda\) and~\(\Lambda'\) are \emph{equivalent}, \(\Lambda\simeq\Lambda'\), if there is \(x\in\LF^\inv\) with \(x\cdot\Lambda=\Lambda'\); they are \emph{adjacent} if there is \(x\in\LF^\inv\) with \(\unif x\cdot\Lambda\subseteq\Lambda'\subseteq x\cdot\Lambda\).  The group~\(\Gl_d(\LF)\) acts on the set of lattices in an obvious way, preserving the relations of equivalence and adjacency.

Vertices of~\(\BuilGl\) are equivalence classes \([\Lambda]\) of such lattices.  The vertices \([\Lambda_1]\), \dots, \([\Lambda_k]\) form a (non-degenerate) simplex in~\(\BuilGl\) if and only if~\(\Lambda_i\) is adjacent but not equivalent to~\(\Lambda_j\) for all \(i,j=1,\dotsc,k\) with \(i\neq j\).  A \(k\)\nb-simplex adjacent to~\([\Lambda]\) is equivalent to a flag \(V_0\subsetneq V_2\subsetneq\dotsb\subsetneq V_k\) in the \(\integ/\maxid\)-vector space \(\integ/\maxid\cdot\Lambda\).

Since any lattice is of the form \(g\cdot\integ^d\) for some \(g\in\Gl_d(\LF)\), the group~\(\Gl_d(\LF)\) acts transitively on the set of vertices.  The stabiliser of \([\integ^d]\) is \(Z\cdot\Gl_d(\integ)\), where
\[
Z\defeq \{x\cdot 1_d\mid x\in\LF^\inv\}\cong \Gl_1(\LF)
\]
denotes the centre of~\(\Gl_d(\LF)\).  An equivariant system of idempotents \((\Proj_{[\Lambda]})_{\Lambda}\) is already specified by the single idempotent \(\Proj\defeq \Proj_{[\integ^d]}\).  This projection must commute with~\(\Gl_d(\integ)\), and any \(\Gl_d(\integ)\)\nb-equivariant idempotent endomorphism~\(\Proj\) of~\(V\) generates an equivariant system of idempotent endomorphisms by~\(\Proj_{[g\integ^n]} \defeq g\Proj g^{-1}\).  It remains to analyse when this equivariant system of idempotent endomorphisms is consistent.  In all applications we care about, \(\Proj\) comes from an idempotent in \(\Hecke\bigl(\Gl_d(\integ)\bigr)\) and the consistency conditions already hold in the ring~\(\Hecke\bigl(\Gl_d(\LF)\bigr)\).  We assume this from now on.

Apartments in the building~\(\BuilGl\) correspond to unordered bases \(\{b_1,\dotsc,b_d\}\) in~\(\LF^d\); the vertices of the apartment for the basis \(\{b_1,\dotsc,b_d\}\) are the lattices of the form \(\sum_{j=1}^d \maxid^{n_j}b_j\) for \(n_1,\dotsc,n_d\in\Z\).  The inequalities \(n_1\ge n_2\ge \dotsb \ge n_d\) define a positive chamber in this apartment.  Let~\(D^+\) be the set of all diagonal matrices in~\(\Gl_d\) in the chosen basis with entries \((\unif^{n_1},\dotsc,\unif^{n_d})\) with \(n_1\ge n_2\ge \dotsb \ge n_d\).  Then the vertices of this positive chamber are \([g\integ^n]\) with \(g\in D^+\).

Since the consistency conditions are compatible with the group actions, it suffices to check Conditions \ref{consistent_commute} and~\ref{consistent_convex} in the special case \(x=[\integ^n]\).  We may also assume that~\(y\) belongs to the positive chamber in the apartment associated to the standard basis of~\(\LF^d\), that is, \(y=\bigl[\sum_{j=1}^d \maxid^{n_j}b_j\bigr]\) with \(n_1,\dotsc,n_d\in\Z\) and \(n_1\ge n_2\ge \dotsb \ge n_d\), because the \(\Gl_d(\integ)\)-orbit of~\(y\) contains a lattice of this form.  The vertex \(y=\bigl[\sum_{j=1}^d \maxid^{n_j}b_j\bigr]\) is adjacent to \([\integ^n]\) if and only if \(n_1-n_d=1\) or, equivalently, we are dealing with the lattice \([\Omega_l\integ^d]\), where~\(\Omega_l\) is the diagonal matrix with~\(l\) entries~\(\unif\) and \(d-l\) entries~\(1\) for some~\(l\).  Thus~\ref{consistent_commute} amounts to
\begin{equation}
  \label{eq:consistent_commute_Gld}
  \Proj \Omega_l\Proj\Omega_l^{-1} = \Omega_l\Proj\Omega_l^{-1} \Proj\qquad
  \text{for \(l=1,\dotsc,d-1\)}.
\end{equation}
Actually, it suffices to establish this for \(l\le \nicefrac{d}{2}\) because if \(l>\nicefrac{d}{2}\) there is \(g\in\Gl_d\) with \(g\integ^n= \Omega_{d-l}\integ^n\) and \(g\Omega^l\integ^n= \unif\cdot\integ^n\).

Condition~\eqref{eq:consistent_commute_Gld} holds if~\(\Proj\) is supported in the normal subgroup \(1+\Mat_d(\maxid)\) of~\(\Gl_d(\integ)\).  Then \(\Omega_l\Proj\Omega_l^{-1}\) is supported in~\(\Gl_d(\integ)\) for all~\(l\).  Thus \(\Omega_l\Proj\Omega_l^{-1}\) commutes with~\(\Proj\) because the latter is assumed central in~\(\Gl_d(\integ)\).  It is unclear whether there is an idempotent~\(\Proj\) not supported in \(1+\Mat_d(\maxid)\) that satisfies~\eqref{eq:consistent_commute_Gld}.

If~\(z\) is a vertex in \(\Hull(x,y)\), then~\(z\) belongs to the same positive chamber in the same apartment, that is, \(z=\bigl[\sum_{j=1}^d \maxid^{m_j}b_j\bigr]\) with \(m_1,\dotsc,m_d\in\Z\) and \(m_1\ge m_2\ge \dotsb \ge m_d\).  It belongs to~\(\Hull(x,y)\) if \(m_i-m_j \le n_i-n_j\) for \(1\le i\le j\le d\).  Equivalently, there are \(g,h\in D^+\) with \(z=h[\integ^n]\) and \(y=gz\).  Therefore, the condition \(\Proj_x\Proj_z\Proj_y = \Proj_x\Proj_y\) for all \(x,y,z\in\BuilGl^\circ\) with \(z\in\Hull(x,y)\) is equivalent to
\begin{equation}
  \label{eq:consistent_convex_Gld}
  \Proj g\Proj\cdot \Proj h\Proj = \Proj gh\Proj\qquad
  \text{for all \(g,h\in D^+\)},
\end{equation}
that is, the map \(D^+\to\Endo(V)\), \(g\mapsto \Proj g\Proj\) is multiplicative.  We may restrict here to~\(z\) adjacent to~\(x\), that is, \(h=\Omega_l\) for some~\(l\) because these elements generate the monoid~\(D^+\).

Equation~\eqref{eq:consistent_convex_Gld} for special projections is frequently used in representation theory; for instance, see~\cite{Casselman:Representations_padic_groups}*{Lemma 4.1.5}.  In particular, it is well-known and easy to check that the idempotent \(\av{U^{(r)}}\) associated to the compact open subgroup
\[
U^{(r)}\defeq 1 + \Mat_d(\maxid^{r+1})
\qquad\text{for } r\in\Z_{\geq 0}
\]
satisfies these conditions.  These subgroups are pro-\(p\)-groups, so that~\(\av{U^{(r)}}\) belongs to \(\Hecke\bigl(\Gl_d,\Z[\nicefrac1p]\bigr)\), and they are normal in~\(\Gl_d(\integ)\) and contained in \(1+\Mat_d(\maxid^{r+1})\), the stabiliser of the star of \([\integ^d]\).  We have already remarked above that this is enough to get a system of open subgroups~\((\Proj_x)_{x\in\BuilGl^\circ}\) satisfying \ref{consistent_group_equivariant} and~\ref{consistent_group_commute}.  It is known also that \(U^{(r)}gU^{(r)}hU^{(r)} = U^{(r)}ghU^{(r)}\) for \(g,h\in D^+\).  This implies~\eqref{eq:consistent_convex_Gld} and hence~\ref{consistent_group_convex}, that is, we have a consistent equivariant system of subgroups~\((U^{(r)}_x)_{x\in\BuilGl^\circ}\).

More explicitly, these subgroups for vertices are
\[
U^{(r)}_{[\Lambda]} \defeq \{g\in \Gl_d(\LF)\mid (g-1)(\Lambda) \subseteq \maxid^{r+1} \cdot\Lambda\}
\]
because this system of subgroups is equivariant and \(U^{(r)}_{[\integ^d]}=U^{(r)}\).  Let~\(\sigma\) be an \(l\)\nb-simplex in~\(\Buil\), let \(\Lambda_0\supsetneq\Lambda_1\supsetneq\dotsb\supsetneq\Lambda_{l-1}\supseteq \Lambda_l=\maxid\cdot\Lambda_0\) be lattices in~\(\LF^d\) that represent its vertices.  Then
\[
U^{(r)}_\sigma = \{g\in\Gl_d(\LF) \mid \text{\((g-1)(\Lambda_{j-1})\subseteq \maxid^r \cdot\Lambda_j\) for \(j=1,\dotsc,l\)}\}.
\]
To check this, notice first that this system of subgroups is \(\Gl_d(\LF)\)\nb-equivariant, so that it suffices to treat one representative in each orbit.  We pick the representatives \(\Lambda_j= \Omega_{k_j}\cdot\integ^d\) with \(0=k_0<k_1<\dotsb<k_l=d\).  In this case, \((g_{ij})\) belongs to \(\prod_{j=0}^{l-1} \Omega_lU^{(r)}\Omega_l^{-1}\) if and only if \(g_{ij}-\delta_{ij}\in\maxid^r \) for \(i\le k_n<j\) for some~\(n\) and \(g_{ij}-\delta_{ij}\in\maxid^{r+1} \) otherwise.  This yields exactly~\(U^{(r)}_\sigma\).

\smallskip

Finally, we let~\(\G_\LF\) be a semi-simple \(p\)\nb-adic group (the generalisation to reductive groups is easy but complicates notation).  The following situation is considered in the theory of types (see~\cite{Bushnell-Kutzko:Types}).

Let \(x\in\Buil^\circ\) be a vertex.  Let \(\StabP_x\subseteq\G_\LF\) be its stabiliser; this is a compact open subgroup of~\(\G_\LF\) because~\(\G\) is semi-simple.  Let~\(\rho\) be an irreducible representation of~\(\StabP_x\); assume that the central projection~\(\Proj_x\) in \(\Hecke(\StabP_x,\Q)\) associated to~\(\rho\) acts on~\(V\) (this is the case if~\(V\) is a \(\Q\)\nb-vector space or if \(\Proj_x\in \Hecke\bigl(\StabP_x,\Z[\nicefrac{1}{p}]\bigr)\) and~\(V\) is a \(\Z[\nicefrac{1}{p}]\)-module).

We view \(\Hecke(\StabP_x)\) as a subalgebra of~\(\Hecke\).  Assume that \(\Proj_xg\Proj_x=0\) if \(g\in\G_\LF\) and \(gx\neq x\).  This ensures that \(\Proj_{gx}\defeq g\Proj_xg^{-1}\) for \(g\in\G_\LF\) and \(\Proj_y=0\) for other vertices defines an equivariant consistent system of idempotents with \(\Proj_\sigma=0\) for all polysimplices of dimension at least~\(1\).  The conditions of Proposition~\ref{pro:consistent} are obvious here because \(\Proj_\sigma\Proj_\tau=0\) unless \(\sigma=\tau\) and \(\dim\sigma=0\).

The cellular chain complex \(\Cell(\Sigma,\Gamma)\) is concentrated in dimension~\(0\) and has
\[
\Ho_0(\Sigma,\Gamma)
= \Cell[0](\Sigma,\Gamma)
\cong \cInd_{\StabP_x}^{\G_\LF} \Proj_x(V),
\]
where \(\cInd\) denotes the compactly supported induction functor:
\[
\cInd_{\StabP_x}^{\G_\LF} \Proj_x(V) =
\{f\colon \G_\LF\to \Proj_x(V) \mid \text{\(\supp f\) is compact}\}^{\StabP_x},
\]
where~\(\StabP_x\) acts by \((g\cdot f)(x) \defeq \pi_g f(xg)\) for \(x\in\G_\LF\), \(g\in\StabP_x\), and~\(\G_\LF\) acts by \((g\cdot f)(x) \defeq f(g^{-1}x)\).  Thus the first half of Theorem~\ref{the:main_theorem_resolution} asserts that \(\cInd_{\StabP_x}^{\G_\LF} \Proj_x(V)\) is isomorphic to the subrepresentation of~\(V\) generated by \(\Proj_x(V)\); for the sheaf cohomology, we get
\[
\Ho^0(\Sigma,\hat{\Gamma})
= \DCell[0](\Sigma,\hat{\Gamma})
\cong \Ind_{\StabP_x}^{\G_\LF} \Proj_x(V),
\]
where \(\Ind\) denotes the induction functor without support restrictions (and without smoothening).  Notice that~\(\G_\LF\) acts roughly and not smoothly on \(\Ho^0(\Sigma,\hat{\Gamma})\).

\subsection{Support projections}
\label{sec:support_projections}

This section prepares the proof of Theorem~\ref{the:main_theorem_resolution} by computing support projections for certain finite subcomplexes of the building.  These projections are interesting in their own right and will be used in Section~\ref{sec:Serre}.  We fix~\(V\) and a consistent system of idempotents \((\Proj_x)\) in \(\Endo(V)\).

\begin{definition}
  \label{def:support_projection}
  Let~\(\Sigma\) be a subcomplex of the building.  A \emph{support projection} for~\(\Sigma\) is an idempotent element \(\support_\Sigma\in\Endo(V)\) with
  \[
  \im(\support_\Sigma) = \sum_{x\in\Sigma^\circ} \im(\Proj_x),\qquad
  \ker(\support_\Sigma) = \bigcap_{x\in\Sigma^\circ} \ker(\Proj_x).
  \]
\end{definition}

Since \(\im(p)\oplus \ker(p)=V\) for any idempotent endomorphism~\(p\) of~\(V\), a support projection exists if and only if
\[
V = \sum_{x\in\Sigma^\circ} \im(\Proj_x) \oplus \bigcap_{x\in\Sigma^\circ} \ker(\Proj_x),
\]
and it is unique if it exists.  It is clear that \(\support_{\Sigma}\le\support_{\Sigma'}\) for \(\Sigma\subseteq\Sigma'\).

It is not clear whether a support projection exists for general~\(\Sigma\).  If, say, \(p\) and~\(q\) are two rank~\(1\) idempotent \(2\times2\)-matrices with \(\ker p=\ker q\) but \(\im p\neq \im q\), then there is no idempotent \(2\times2\)-matrix with kernel \(\ker p\bigcap\ker q\) and image \(\im p+\im q\).  For a set \(\{p_i\}_{i\in I}\) of self-adjoint projections on Hilbert space, there is a unique self-adjoint projection with kernel \(\bigcap \ker(p_i)\); but its image is the \emph{closure} of \(\sum \im(p_i)\), and there is no simple formula that expresses it using the given projections~\(p_i\).

We are going to show that the support projection of a finite convex subcomplex exists and is given by a straightforward formula.  Our method applies to more general subcomplexes of the building.  To understand the necessary and sufficient condition for this, we first need two geometric lemmas about hulls.

\begin{lemma}
  \label{lem:minimal_face}
  Let~\(\sigma\) and~\(x\) be a polysimplex and a vertex in~\(\Buil\).  Then there is a unique minimal face~\(\tau\) of~\(\sigma\) with \(\sigma\in\Hull(x,\tau)\).  That is, a face~\(\omega\) of~\(\sigma\) satisfies \(\sigma\in\Hull(x,\omega)\) if and only if \(\omega\succ\tau\).
\end{lemma}
  \begin{figure}[htbp]
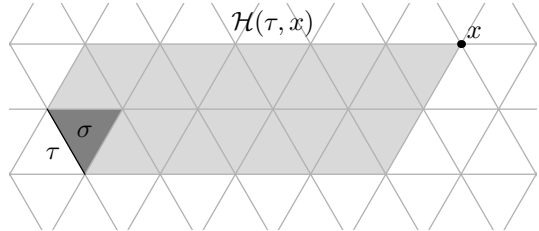

    \centering
\begin{asy}
unitsize(.5cm);
pair z[];//path of points in the building
//vertices of simplex tau
z[0]=2*a[1];
z[1]=(2*a[1]+a[2]);
z[2]=(2*a[1]+a[1]);
//vertex x
z[3]=(4*a[1]+4*a[0]);

//further vertices of convex hull
z[4]=(2*a[1]+4*a[0]);
z[5]=4*a[1]-a[0];
z[6]=3*a[1]-a[0];

//draw the convex hull of x and tau
fill(z[0]--z[4]--z[3]--z[5]--z[6]--cycle,white*.85);
//draw sigma
fill(z[0]--z[1]--z[2]--cycle,white*.5);

//draw the walls of the building:
for(i=-18;i<30;++i){
  draw((i*a[0])--(i*a[0]+20*a[1]),white*.7);
  draw((i*a[0])--(i*a[0]+20*a[2]),white*.7);
}
for(i=0;i<20;++i){
  draw((i*a[2]-10*a[0])--(i*a[2]+(20+i)*a[0]),white*.7);
}

//clip to rectangular region:
clip((0,2)--(14,2)--(14,8)--(0,8)--cycle);

//draw and label tau
draw("\(\tau\)", z[0]--z[1],SW);
dot("\(x\)",z[3],NE);
//label sigma
label("\(\sigma\)",.3333*z[0]+.333*z[1]+.333*z[2]);

//label region
label("\(\mathcal{H}(\tau,x)\)",.5*z[3]+.5*z[5],N);
\end{asy}
    \caption{Illustration of Lemma~\ref{lem:minimal_face} and the first half of Lemma~\ref{lem:maximal_cone}}
    \label{fig:minimal_face}
  \end{figure}
\begin{proof}
  Fix an apartment~\(\Apart\) containing~\(\sigma\) and~\(x\).  If the underlying affine root system is reducible, then we split \(\Apart=\prod_{i=1}^n \Apart_i\) and \(\sigma= \prod_{i=1}^n \sigma_i\).  If \(\tau_i\subseteq\sigma_i\) solves the problem in~\(\Apart_i\), then the polysimplex \(\tau\defeq \prod \tau_i\) 
solves it in~\(\Apart\).  Hence we may assume that~\(\Apart\) is irreducible.

  Let~\(\Delta\) be a chamber containing~\(\sigma\) and let \(\affRoot_0,\dotsc,\affRoot_d\) be the corresponding simple affine roots with \(\Delta = \bigcap_{j=0}^d \affRoot_j^{\ge0}\).  If there is~\(j\) with \(\affRoot_j|_\sigma=0\) and \(\affRoot_j(x)<0\), then we reflect~\(\Delta\) at the corresponding wall.  The new chamber has fewer~\(j\) with \(\affRoot_j|_\sigma=0\) and \(\affRoot_j(x)>0\).  After finitely many steps, we achieve that \(\affRoot_j(x)\ge0\) for all~\(j\) with \(\affRoot_j|_\sigma=0\).  Faces of~\(\Delta\) correspond to subsets~\(I\) of \(\{0,\dotsc,d\}\) via \(I\mapsto \Delta\cap\bigcap_{j\in I} \ker(\affRoot_j)\).  This yields a face of~\(\sigma\) if \(j\in I\) for all~\(j\) with \(\affRoot_j|_\sigma=0\).  Let~\(I\) be the subset of all~\(j\) with \(\affRoot_j(x)>0\) or \(\affRoot_j|_\sigma=0\).  We claim that the corresponding face~\(\tau\) of~\(\sigma\) satisfies \(\sigma\in\Hull(\tau,x)\) and is minimal with this property.

  Let \(\omega\prec\sigma\) satisfy \(\affRoot_j|_\omega=0\) for some \(j\notin I\).  Then \(\affRoot_j(x)\le0\) and hence \(\affRoot_j|_{\Hull(x,\omega)}\le0\), so that \(\sigma\notin\Hull(x,\omega)\).  Therefore, if \(\sigma\in\Hull(x,\omega)\) then \(\tau\prec\omega\).  Conversely, we claim that \(\sigma\in\Hull(x,\tau)\).  Let~\(\varaffRoot\) be any affine root with \(\varaffRoot(x)\ge0\) and~\(\varaffRoot|_\tau\ge0\).  We must show \(\varaffRoot|_\sigma\ge0\).  This is clear if~\(\varaffRoot|_\tau>0\) or \(\varaffRoot|_\sigma=0\), so that we may assume that~\(\varaffRoot\) vanishes on~\(\tau\) but not on~\(\sigma\).  We have \(\varaffRoot=\sum_{j=0}^d \lambda_j\affRoot_j\) with coefficients~\(\lambda_j\) of the same sign.  Since \(\varaffRoot|_\tau=0\), we have \(\lambda_j=0\) for \(j\notin I\).  Since~\(\varaffRoot|_\sigma\neq0\), some~\(\lambda_j\) with \(\affRoot_j|_\sigma\neq0\) is non-zero.  Since \(\affRoot_j(x)>0\) for this~\(j\) and \(\affRoot_k(x)\ge0\) all \(k\in I\), we have \(\lambda_j\ge0\) for \(j\in I\).  Hence \(\varaffRoot|_\sigma\ge0\).
\end{proof}

\begin{lemma}
  \label{lem:maximal_cone}
  Let~\(\tau\) and~\(x\) be a polysimplex and a vertex in~\(\Buil\).  Then there is a unique maximal polysimplex \(\sigma\in\Hull(x,\tau)\) with \(\tau\prec\sigma\) \textup(see also Figure~\textup{\ref{fig:minimal_face})}.  That is, a polysimplex~\(\omega\) satisfies \(\omega\in\Hull(x,\tau)\) and \(\tau\prec\omega\) if and only if \(\tau\prec\omega\prec\sigma\).

  Moreover, if \(\sigma=\tau\) then there is a proper face~\(\omega\) of~\(\tau\) with \(\tau\in\Hull(x,\omega)\).
\end{lemma}

\begin{proof}
  Let \(\omega_1\) and~\(\omega_2\) be two polysimplices contained in \(\Hull(x,\tau)\) and containing~\(\tau\).  We must show that they are adjacent, that is, they are both faces of a polysimplex~\(\omega\).  If not, then they are separated by an affine root~\(\affRoot\), say, \(\affRoot|_{\omega_1}>0\) and \(\affRoot|_{\omega_2}<0\).  This implies \(\affRoot|_\tau=0\).  If \(\affRoot(x)\ge0\), then~\(\affRoot\) separates \(x\) and~\(\tau\) from~\(\omega_2\), contradicting \(\omega_2\in\Hull(x,\tau)\).  If \(\affRoot(x)<0\), then~\(\affRoot\) separates \(x\) and~\(\tau\) from~\(\omega_1\), contradicting \(\omega_1\in\Hull(x,\tau)\).  Hence \(\omega_1\) and~\(\omega_2\) are adjacent.  Of course, \([\omega_1,\omega_2]\) is still contained in \(\Hull(x,\tau)\).

  Now assume \(\sigma=\tau\).  Let \(\varphi\colon [0,1]\to\Sigma\) be the geodesic between~\(x\) and an interior point of~\(\tau\).  The points \(\varphi(t)\) for \(t\approx1\) belong to a polysimplex~\(\sigma'\) with \(\sigma'\in\Hull(x,\tau)\) and \(\tau\prec\sigma'\).  Since \(\sigma=\tau\), we have \(\varphi(t)\in\abs{\tau}\) for \(t\approx1\).  Then we may prolong~\(\varphi\) to a geodesic beyond \(\varphi(1)\) until it hits a proper face~\(\omega\) of~\(\tau\).  Since \(\Hull(\omega,x)\) contains~\(\varphi(1)\), an interior point of~\(\tau\), we get \(\tau\in\Hull(\omega,x)\) for some \(\omega\prec\tau\) with \(\omega\neq\tau\).
\end{proof}

Our criterion for support projections requires an analogue of Lemma~\ref{lem:maximal_cone} for the subcomplex~\(\Sigma\) instead of~\(\Buil\):

\begin{definition}
  \label{def:admissible}
  A subcomplex \(\Sigma\subseteq\Buil\) is called \emph{admissible} if it has the following two properties:
  \begin{itemize}
  \item For any polysimplex \(\tau\in\Buil\), \(\Sigma\cap\tau\) is again a polysimplex or empty.

  \item Let \(x\in\Sigma^\circ\) and \(\tau\in\Sigma\).  If \(\tau\neq x\) and~\(\tau\) has no proper face~\(\omega\) with \(\tau\in\Hull(x,\omega)\), then~\(\tau\) is a proper face of a polysimplex in \(\Sigma \cap \Hull(x,\tau)\).
  \end{itemize}
\end{definition}
The first condition is equivalent to the following requirement:
if \(x_1,\dotsc,x_n\) are adjacent vertices in~\(\Sigma\),
then~\(\Sigma\) contains the polysimplex \([x_1,\dotsc,x_n]\)
that they span.  Thus admissible subcomplexes are determined by
the vertices they contain.

Figures \ref{fig:admissible} and~\ref{fig:forbidden_admissible}
illustrate admissible subcomplexes.  The first figure shows an
example of an admissible subcomplex and two examples of
non-admissible subcomplexes that violate the first and second
condition in Definition~\ref{def:admissible}, respectively.
Figure~\ref{fig:forbidden_admissible} illustrates the second
condition.  If an admissible subcomplex of an
\(\tilde{A}_2\)\nb-apartment contains~\(a\) but not \(b\)
and~\(c\), then it may not contain any points in the first
forbidden region.  If it contains \(x\) and~\(y\) but
not~\(z\), then it may not contain any points in the second
forbidden region.

\begin{figure}[htbp]
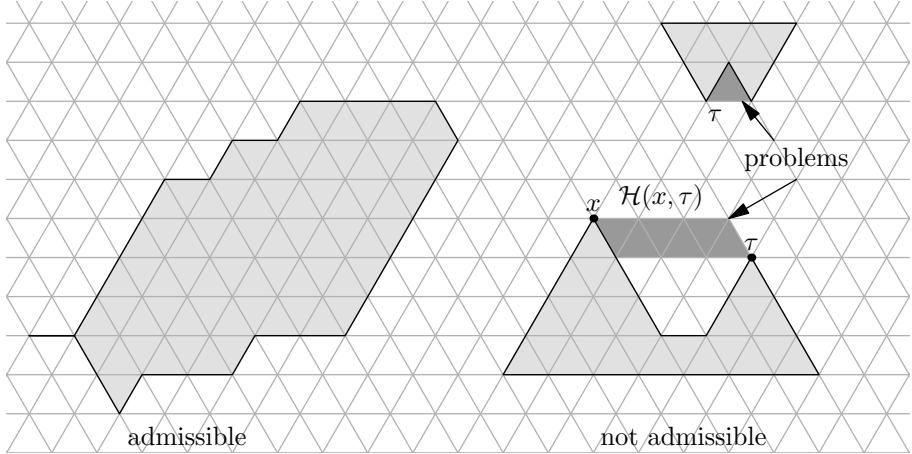

  \centering
\begin{asy}

path admissible;
path nonadmissible1;
path nonadmissible2;

unitsize(.3cm);
pair z[];//path of points in the building
z[0]=2a[0]+a[1];
z[1]=z[0]+2a[2];
z[15]=z[1]-a[0];
z[2]=z[1]+4a[1];
z[3]=z[2]+a[0];
z[4]=z[3]+a[1];
z[5]=z[4]+a[0];
z[6]=z[5]+a[1];
z[7]=z[6]+3a[0];
z[8]=z[7]-a[2];
z[9]=z[8]-5a[1];
z[10]=z[9]-2a[0];
z[11]=z[10]-a[1];
z[12]=z[11]-2a[0];
admissible=(z[0]--z[1]--z[15]--z[1]--z[2]--z[3]--z[4]--z[5]--z[6]--z[7]--z[8]--z[9]--z[10]--z[11]--z[12]--cycle);

z[20]=10a[0]+2a[1];
z[21]=z[20]+7a[0];
z[22]=z[21]+3a[2];//tau
z[23]=z[22]-2a[1];
z[24]=z[23]-a[0];
z[25]=z[24]+3a[2];//x
nonadmissible1=z[20]--z[21]--z[22]--z[23]--z[24]--z[25]--cycle;

z[30]=11a[0]+9a[1];
z[31]=z[30]+a[1];
z[32]=z[31]-a[2];
z[33]=z[32]+2a[1];
z[34]=z[33]-3a[0];
nonadmissible2=z[30]--z[31]--z[32]--z[33]--z[34]--cycle;

//fill all relevant regions first
fill(admissible, white*.88);//admissible region
fill(nonadmissible1, white*.88);//first non-admissible region
fill(nonadmissible2, white*.88);//second non-admissible region
fill(z[22]--(z[22]+a[2])--z[25]--(z[25]-a[2])--cycle, white*.60);//problematic region
fill(z[30]--z[31]--z[32]--cycle, white*.60);//problematic region

//draw the walls of the building:
for(i=-18;i<30;++i){
  draw((i*a[0])--(i*a[0]+20*a[1]),white*.7);
  draw((i*a[0])--(i*a[0]+20*a[2]),white*.7);
}
for(i=0;i<20;++i){
  draw((i*a[2]-10*a[0])--(i*a[2]+(20+i)*a[0]),white*.7);
}
//clip walls to rectangular region:
clip((0,0)--(40,0)--(40,20)--(0,20)--cycle);

// draw boundaries of admissible and non-admissible regions
draw(admissible);
draw(nonadmissible1);
draw(nonadmissible2);

//label special points and domains
dot("$x$",z[25],N);
dot("$\tau$",z[22],N);
label("$\mathcal{H}(x,\tau)$",z[25]+1.5a[0],N);
label("$\tau$",z[30]+.2a[0],S);

//indicate problems
draw(z[22]+2a[1]..z[22]+a[2],Arrow);
draw(z[22]+2a[1]+a[2]..z[32]-.2a[0],Arrow);
label("problems",z[22]+2a[1],N);

//legend
label("admissible",4a[0],N);
label("not admissible",15a[0],N);
\end{asy}
  \caption{Admissible and non-admissible subcomplexes of an \(\tilde{A}_2\)\nb-apartment}
  \label{fig:admissible}
\end{figure}

\begin{figure}[htbp]
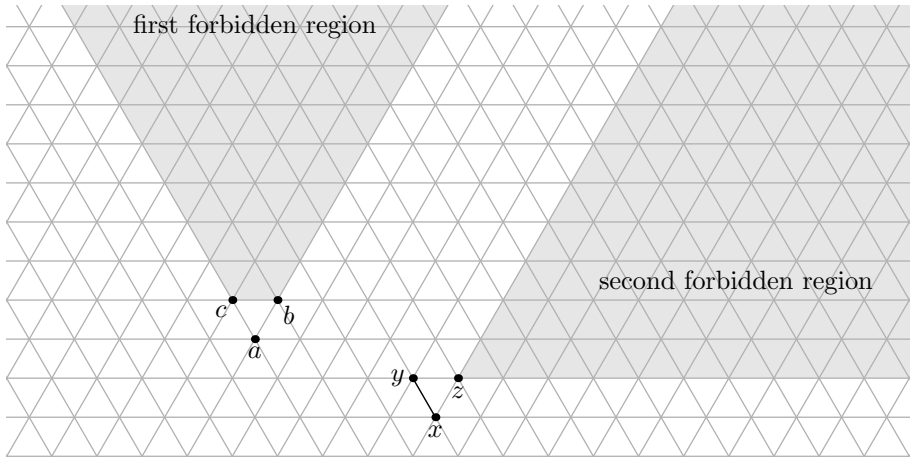

\begin{asy}
unitsize(.3cm);
pair z[];//various points in the building
z[0]=9a[0]+a[1];//x
z[1]=z[0]+a[2];//y
z[2]=z[0]+a[1];//z
z[3]=4a[0]+3a[1];//a
z[4]=z[3]+a[1];//b
z[5]=z[3]+a[2];//c

fill(z[2]+20a[0]--z[2]--z[2]+30a[1]--cycle, white*.9);//forbidden region 1
fill(z[3]+20a[2]--z[5]--z[4]--z[3]+20a[1]--cycle, white*.9);//forbidden region 2

//draw the walls of the building:
for(i=-18;i<30;++i){
  draw((i*a[0])--(i*a[0]+20*a[1]),white*.7);
  draw((i*a[0])--(i*a[0]+20*a[2]),white*.7);
}
for(i=0;i<20;++i){
  draw((i*a[2]-10*a[0])--(i*a[2]+(20+i)*a[0]),white*.7);
}

draw(z[0]--z[1]);
dot("$x$",z[0],S);
dot("$y$",z[1],W);
dot("$z$",z[2],S);
dot("$a$",z[3],S);
dot("$b$",z[4],SE);
dot("$c$",z[5],SW);
label("second forbidden region",z[2]+2a[0]+2a[1],NE);
label("first forbidden region",z[3]+4a[1]+4a[2]);

//clip to rectangular region:
clip((0,0)--(40,0)--(40,20)--(0,20)--cycle);
\end{asy}

  \caption{Forbidden regions for admissible subcomplexes of an apartment}
  \label{fig:forbidden_admissible}
\end{figure}

\begin{lemma}
  \label{lem:convex_admissible}
  Convex subcomplexes are admissible.  If~\(\Sigma_1\) is convex and~\(\Sigma_2\) is admissible, then \(\Sigma_1\cap\Sigma_2\) is admissible.
\end{lemma}

\begin{proof}
  Since an intersection of two polysimplices is again a polysimplex or empty, the first property is hereditary for intersections.  Moreover, it is trivial for convex subcomplexes.  The second property is inherited by intersections with convex subcomplexes because all \(\omega\in\Sigma_2\cap \Hull(x,\tau)\) belong to~\(\Sigma_1\) if \(x\in\Sigma_1^\circ\) and \(\tau\in\Sigma_1\).

  We check that a convex subcomplex~\(\Sigma\) satisfies the second condition for admissible subcomplexes.  Let \(x\in\Sigma^\circ\) and \(\tau\in\Sigma\).  Let \(\varphi\colon [0,1]\to\Sigma\) be the geodesic between~\(x\) and an interior point of~\(\tau\).  If \(\varphi (1-\varepsilon)\in\tau\) for sufficiently small \(\varepsilon>0\), we may prolong~\(\varphi\) to a geodesic beyond \(\varphi(1)\) until it hits a face~\(\omega\) of~\(\tau\).  Then \(\Hull(\omega,x)\) contains an interior point of~\(\tau\), so that \(\tau\in\Hull(\omega,x)\) for some \(\omega\prec\tau\) with \(\omega\neq\tau\).  If \(\varphi (1-\varepsilon)\notin\tau\) for all \(\varepsilon>0\), let~\(\omega\) be the minimal polysimplex containing~\(\varphi(1-\varepsilon)\) for all sufficiently small \(\varepsilon>0\).  Then~\(\tau\) is a proper face of~\(\omega\), and \(\omega\in\Hull(\tau,x)\) because an interior point of~\(\omega\) lies on a geodesic between an interior point of~\(\tau\) and~\(x\).
\end{proof}

\begin{theorem}
  Let\label{the:support_projection}~\((\Proj_x)\) be a consistent system of idempotents in \(\Endo(V)\) and let~\(\Sigma\) be a finite convex subcomplex of \(\Buil\) or, more generally, a finite admissible subcomplex of \(\Buil\).  Then
   \[
   \support_\Sigma \defeq \sum_{\sigma\in\Sigma} (-1)^{\dim(\sigma)} \Proj_\sigma
   \]
  is the support projection for~\(\Sigma\).
\end{theorem}

It is remarkable that this simple formula for~\(\support_\Sigma\) works although the idempotents~\(\Proj_\sigma\) do not commute.

\begin{proof}
  Define~\(\support_\Sigma\) by the above formula.  Since \(\Proj_\sigma\ge\Proj_\tau\) for \(\sigma\prec\tau\), we clearly have
  \[
  \im(\support_\Sigma) \subseteq \sum_{\sigma\in\Sigma} \im(\Proj_\sigma) = \sum_{x\in\Sigma^\circ} \im(\Proj_x),\qquad
  \ker(\support_\Sigma) \supseteq \bigcap_{\sigma\in\Sigma} \ker(\Proj_\sigma) = \bigcap_{x\in\Sigma^\circ} \ker(\Proj_x).
  \]
  We will prove \(\Proj_x\support_\Sigma = \Proj_x = \support_\Sigma\Proj_x\) for all \(x\in\Sigma^\circ\).  This implies \(\Proj_\sigma\support_\Sigma = \Proj_\sigma = \support_\Sigma\Proj_\sigma\) for all \(\sigma\in\Sigma\) using the definition of~\(\Proj_\sigma\) in Proposition~\ref{pro:consistent}, and then
  \[
  \support_\Sigma\cdot\support_\Sigma
  = \sum_{\sigma\in\Sigma} (-1)^{\abs{\sigma}} \Proj_\sigma \support_\Sigma
  = \sum_{\sigma\in\Sigma} (-1)^{\abs{\sigma}} \Proj_\sigma
  =\support_\Sigma.
  \]
  Furthermore, it follows that
  \[
  \im(\support_\Sigma) \supseteq \sum_{x\in\Sigma^\circ} \im(\Proj_x),\qquad
  \ker(\support_\Sigma) \subseteq \bigcap_{x\in\Sigma^\circ} \ker(\Proj_x),
  \]
  so that~\(\support_\Sigma\) is the support projection of~\(\Sigma\).  Thus it remains to establish \(\Proj_x\support_\Sigma = \Proj_x = \support_\Sigma\Proj_x\) for all \(x\in\Sigma^\circ\).  We only write down the proof of \(\Proj_x\support_\Sigma=\Proj_x\); the other equation is obtained by working in the opposite category.

  Let \(m(\sigma)\) for a polysimplex~\(\sigma\) be the minimal face~\(\tau\) of~\(\sigma\) with \(\sigma\in \Hull(x,\tau)\).  This map is idempotent, that is, \(m(\sigma)\) has the property that \(m(\sigma)\notin \Hull(x,\tau)\) for any proper face~\(\tau\) of~\(m(\sigma)\) because otherwise \(\sigma\in \Hull\bigl(x,m(\sigma)\bigr) = \Hull(x,\tau)\).  Let \(M\subseteq\Sigma\) be the set of all polysimplices of the form \(m(\sigma)\).  The consistency conditions in Proposition~\ref{pro:consistent} imply \(\Proj_x\Proj_\sigma = \Proj_x\Proj_\sigma\Proj_{m(\sigma)} = \Proj_x\Proj_{m(\sigma)}\).  Hence we may rewrite
  \begin{equation}
    \label{eq:Proj_support_compute}
    \Proj_x\support_\Sigma = \sum_{\tau\in M} \sum_{\twolinesubscript{\sigma\in\Sigma}{m(\sigma)=\tau}}
 (-1)^{\dim(\sigma)} \Proj_x\Proj_\tau.
  \end{equation}

  For each \(\tau\in M\), Lemma~\ref{lem:maximal_cone} and the first admissibility assumption on~\(\Sigma\) yield \(\omega\in\Sigma\) such that the set of \(\sigma\in\Sigma\) with \(m(\sigma)=\tau\) is exactly the set of all \(\sigma\in\Buil\) with \(\tau\prec\sigma\prec\omega\): first construct such a maximal~\(\omega\) in~\(\Buil\), then its intersection with~\(\Sigma\) works.  The second admissibility assumption about~\(\Sigma\) yields \(\omega\neq\tau\) or \(\tau=x\) because \(\tau\in M\).  The alternating sum of the dimensions of all polysimplices~\(\sigma\) with \(\tau\prec\sigma\prec\omega\) vanishes for \(\tau\neq\omega\) and is~\(1\) if \(\tau=\omega\).  For simplicial complexes, this is because such intermediate faces correspond bijectively to subsets of \(\omega^\circ\setminus\tau^\circ\).  For polysimplicial complexes, we use the product decomposition to reduce the assertion to the simplicial case.  Hence the summand for \(\tau\in M\) vanishes unless \(\tau=\omega\), that is, \(\tau=x\).  Thus \(\Proj_x\support_\Sigma=\Proj_x\).
\end{proof}

This Theorem implies several properties of support projections and hence of the subspaces \(\sum_{x\in\Sigma^\circ} \im(\Proj_x)\) and \(\bigcap_{x\in\Sigma^\circ} \ker(\Proj_x)\).

\begin{corollary}
  \label{cor:support_additive}
  Let \(\Sigma_+\) and~\(\Sigma_-\) be two finite subcomplexes and let \(\Sigma_0\defeq \Sigma_+\cap\Sigma_-\) and \(\Sigma=\Sigma_+\cup\Sigma_-\).  Assume that all four subcomplexes \(\Sigma_+\), \(\Sigma_-\), \(\Sigma_0\) and~\(\Sigma\) are admissible.  Then \(\support_\Sigma = \support_{\Sigma_+} + \support_{\Sigma_-} - \support_{\Sigma_0}\), \(\support_{\Sigma_+}\support_{\Sigma_-} = \support_{\Sigma_-}\support_{\Sigma_+} = \support_{\Sigma_0}\), and
  \begin{align*}
  \sum_{x\in\Sigma_+^\circ} \im(\Proj_x) \cap \sum_{x\in\Sigma_-^\circ} \im(\Proj_x)
  &= \sum_{x\in\Sigma_0^\circ} \im(\Proj_x),\\
  \bigcap_{x\in\Sigma_+^\circ} \ker(\Proj_x) + \bigcap_{x\in\Sigma_-^\circ} \ker(\Proj_x)
  &= \bigcap_{x\in\Sigma_0^\circ} \ker(\Proj_x).
  \end{align*}
\end{corollary}

\begin{proof}
  The formula for~\(\support_\Sigma\) follows immediately from Theorem~\ref{the:support_projection}.  Since \(\support_\Sigma\), \(\support_{\Sigma_+}\), and \(\support_{\Sigma_-}-\support_{\Sigma_0}\) are idempotent, it follows that \(\support_{\Sigma_+}\) and \(\support_{\Sigma_-}-\support_{\Sigma_0}\) are orthogonal idempotents, so that \(\support_{\Sigma_+}\support_{\Sigma_-} = \support_{\Sigma_-}\support_{\Sigma_+} = \support_{\Sigma_0}\).  The assertions about subspaces are special cases of assertions about commuting idempotent operators.
\end{proof}

Let \(\Sigma_+\), \(\Sigma_0\) and~\(\Sigma_-\) be finite admissible subcomplexes of~\(\Buil\).  We say that~\(\Sigma_0\) \emph{separates} \(\Sigma_+\) and~\(\Sigma_-\) if there are finite admissible subcomplexes \(\Sigma_+'\) and~\(\Sigma_-'\) with \(\Sigma_\pm\subseteq\Sigma_\pm'\), \(\Sigma_0=\Sigma_+'\cap\Sigma_-'\), and \(\Sigma_+'\cup\Sigma_-'\) admissible.

\begin{corollary}
  \label{cor:support_separate}
  If~\(\Sigma_0\) separates \(\Sigma_+\) and~\(\Sigma_-\), then \(\support_{\Sigma_+}\support_{\Sigma_-} = \support_{\Sigma_+}\support_{\Sigma_0}\support_{\Sigma_-}\).
\end{corollary}

\begin{proof}
  We have \(\support_{\Sigma_\pm} \le \support_{\Sigma_\pm'}\).  By the previous corollary, \(\support_{\Sigma_+'}\support_{\Sigma_-'} = \support_{\Sigma_0}\).  Hence \(\support_{\Sigma_+} \support_{\Sigma_-} = \support_{\Sigma_+} \support_{\Sigma_+'} \support_{\Sigma_-'} \support_{\Sigma_-} = \support_{\Sigma_+} \support_{\Sigma_0} \support_{\Sigma_-}\).
\end{proof}

In particular, this applies to \(\Proj_x=\support_x\) for a vertex~\(x\) of~\(\Buil\), and we get \(\Proj_x\support_\Sigma\Proj_y = \Proj_x\Proj_y\) if~\(\Sigma\) separates \(x\) and~\(y\).

\subsection{Proof of Proposition~\ref{pro:consistent}}
\label{sec:proof_consistent}

This section establishes that consistency conditions for the idempotents \((\Proj_x)_{x\in\Buil^\circ}\) for vertices imply consistency conditions for the idempotents \((\Proj_\sigma)_{\sigma\in\Buil}\) for all polysimplices.  Most of the argument deals with the geometry of the building: we need chains of adjacent vertices or polysimplices in hulls of polysimplices.

Condition~\ref{consistent_commute} in Definition~\ref{def:consistent_idempotents} implies that the order in the product
\[
\Proj_\sigma \defeq \prod_{\twolinesubscript{x\in\Buil^\circ}{x\prec\sigma}} \Proj_x
\]
defining~\(\Proj_\sigma\) does not matter.  Hence~\(\Proj_\sigma\) is a well-defined idempotent endomorphism of~\(V\).  The same argument yields \(\Proj_\sigma\Proj_\tau=\Proj_{[\sigma,\tau]}\) for adjacent polysimplices \(\tau\) and~\(\sigma\).  Condition~\ref{consistent_simplex_equivariant} follows immediately from~\ref{consistent_equivariant}.  We will spend the remainder of this section to check that \ref{consistent_commute} and~\ref{consistent_convex} imply~\ref{consistent_simplex_convex}.  We begin with two geometric lemmas.

\begin{lemma}
  \label{lem:simplex_path}
  Let \(\tau,\sigma,\omega\) be polysimplices in the building with \(\omega\in\Hull(\sigma,\tau)\).  There is a finite sequence of polysimplices \(\tau_0=\tau\), \(\tau_1\), \dots, \(\tau_{m-1}\), \(\tau_m=\omega\) such that \(\tau_i\in\Hull(\omega,\tau_{i-1})\), \(\omega\in\Hull(\sigma,\tau_i)\), and either \(\tau_{i-1}\prec\tau_i\) or \(\tau_{i-1}\succ\tau_i\) for \(i=1,\dotsc,m\) \textup{(}see Figure~\textup{\ref{fig:simplex_path})}.
\end{lemma}

  \begin{figure}[htbp]
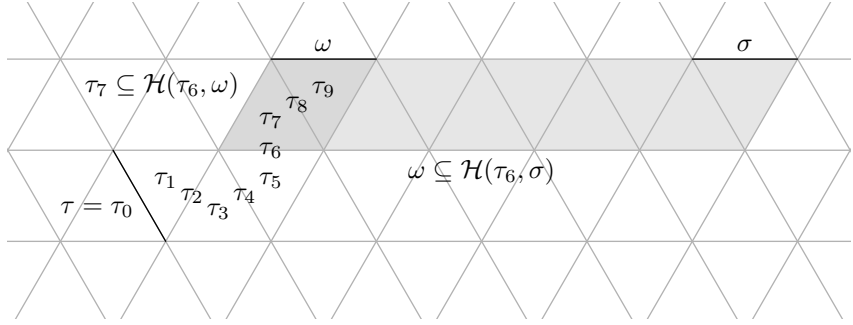

    \centering
\begin{asy}
unitsize(.7cm);
pair z[];//path of points in the building
//vertices of τ
z[0]=2*a[1]+a[0];
z[10]=3*a[1];

//vertices of σ
z[1]=(4*a[1]+6*a[0]);
z[2]=4*a[1]+5*a[0];

//vertices of ω
z[3]=(4*a[1]+a[0]);
z[4]=4*a[1]+2*a[0];

//other corners of hulls of σ,τ
z[7]=4*a[1];
z[8]=2*a[1]+6*a[0];

z[11]=(z[0]+z[10]+3*a[1]+a[0])/3;//τ₁
z[12]=2.5*a[1]+a[0];//τ₂
z[13]=(7*a[1]+4*a[0])/3;//τ₃
z[14]=3*a[1]-.5*a[2]+a[0];//τ₄
z[15]=(8*a[1]+5*a[0])/3;//τ₅
z[16]=3*a[1]+1.5*a[0];//τ₆
z[17]=(10*a[1]+4*a[0])/3;//τ₇
z[18]=3.5*a[1]+1.5*a[0];//τ₈
z[19]=(11*a[1]+5*a[0])/3;//τ₉

//draw the convex hull of τ₆ and ω
fill((3*a[1]+a[0])--(z[1]-a[1])--z[1]--z[3]--cycle,white*.9);
fill((3*a[1]+a[0])--(3*a[1]+2*a[0])--z[4]--z[3]--cycle,white*.85);

//draw the walls of the building:
for(i=-18;i<30;++i){
  draw((i*a[0])--(i*a[0]+20*a[1]),white*.7);
  draw((i*a[0])--(i*a[0]+20*a[2]),white*.7);
}
for(i=0;i<20;++i){
  draw((i*a[2]-10*a[0])--(i*a[2]+(20+i)*a[0]),white*.7);
}

//clip to rectangular region:
clip((1,2)--(17,2)--(17,8)--(1,8)--cycle);

//draw and label σ, ω and τ
draw("\(\tau=\tau_0\)", z[0]--z[10],SW);
draw("\(\sigma\)",z[1]--z[2],N);
draw("\(\omega\)",z[3]--z[4],N);

//draw simplices along the way
label("\(\tau_1\)",z[11]);
label("\(\tau_2\)",z[12]);
label("\(\tau_3\)",z[13]);
label("\(\tau_4\)",z[14]);
label("\(\tau_5\)",z[15]);
label("\(\tau_6\)",z[16]);
label("\(\tau_7\)",z[17]);
label("\(\tau_8\)",z[18]);
label("\(\tau_9\)",z[19]);

//label region
label("\(\tau_7\subseteq\mathcal{H}(\tau_6,\omega)\)",.5*(z[3]+3*a[1]+a[0]),NW);
label("\(\omega\subseteq\mathcal{H}(\tau_6,\sigma)\)",.5*(2*a[1]+a[0]+z[1]),S);
\end{asy}
    \caption{Illustration of Lemma~\ref{lem:simplex_path}}
    \label{fig:simplex_path}
  \end{figure}

\begin{proof}
  Let \(\varphi\colon [0,1]\to\Buil\) be a geodesic between interior points of \(\tau\) and~\(\omega\).  Each \(\varphi(t)\) is an interior point of some polysimplex~\(\tau(t)\).  The function \(t\mapsto\tau(t)\) is piecewise constant.  Let \(0=t_0<t_2<t_4<\dotsb<t_{2n-2}<t_{2n}=1\) be the points where \(\tau(t)\) jumps and choose \(t_1\), \dots, \(t_{2n-1}\) with \(t_0<t_1<t_2<\dotsb<t_{2n-1}<t_{2n}\).  Let \(\tau_i = \tau(t_i)\), so that \(\tau_0=\tau\) and \(\tau_{2n}=\omega\).  Then \(\tau_{2j}\) and~\(\tau_{2j+2}\) must be faces of \(\tau_{2j+1}\) for \(j=0,\dotsc,n-1\), so that either \(\tau_i\prec\tau_{i+1}\) or \(\tau_i\succ\tau_{i+1}\) for \(i=0,\dotsc,2n\).  Since some interior point of~\(\tau_i\) lies on a geodesic between interior points of \(\tau_{i-1}\) and~\(\omega\), we have \(\tau_i\in\Hull(\omega,\tau_{i-1})\).

  It remains to check \(\omega\in\Hull(\sigma,\tau_i)\).  Let~\(\Apart\) be an apartment containing \(\tau\) and~\(\sigma\).  Then~\(\Apart\) also contains~\(\omega\) because \(\omega\in\Hull(\sigma,\tau)\).  Hence~\(\Apart\) contains all polysimplices~\(\tau_i\).  If not \(\omega\in\Hull(\sigma,\tau_i)\), then there is an affine root~\(\affRoot\) on~\(\Apart\) with \(\affRoot|_{\tau_i}\ge0\) and \(\affRoot|_{\sigma}\ge0\), but \(\affRoot|_\omega < 0\).  Since \(\affRoot\circ\varphi(t)= \lambda t + \mu \) for some \(\lambda ,\mu \in \R\) and \(\affRoot\circ\varphi\) changes sign between \(t_i\) and~\(1\), it cannot change sign between \(0\) and~\(t_i\), so that \(\affRoot\circ\varphi(0)\ge0\) as well, that is, \(\affRoot|_\tau\ge0\).  But then~\(\affRoot\) separates~\(\omega\) from \(\tau\cup\sigma\), contradicting \(\omega\in\Hull(\sigma,\tau)\).  Hence \(\omega\in\Hull(\sigma,\tau_i)\).
\end{proof}

\begin{lemma}
  \label{lem:point_path}
  Let \(\sigma\) and~\(\tau\) be polysimplices in~\(\Buil\) and let~\(y\) be a vertex adjacent to~\(\sigma\) with \(y\in\Hull(\sigma,\tau)\).  Then there is a finite sequence of vertices \(z_0\), \dots, \(z_m\) with \(z_m=y\) and \(z_0\prec\tau\) such that~\(z_i\) is adjacent to~\(z_{i-1},\, z_i\in\Hull(y,z_{i-1})\) and \(y\in\Hull(\sigma,z_i)\) for \(i=1,\dotsc,m\) \textup{(}see Figure~\textup{\ref{fig:point_path})}.  In particular, there is a vertex~\(z\) of~\(\tau\) with \(y\in\Hull(\sigma,z)\).
\end{lemma}
  \begin{figure}[htbp]
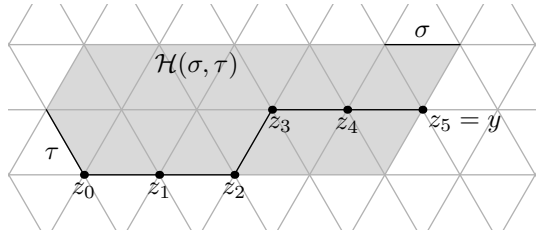

    \centering
\begin{asy}
unitsize(.5cm);
pair z[];//path of points in the building
z[0]=2*a[1];//z₀∈τ
z[1]=(2*a[1]+1*a[0]);
z[2]=(2*a[1]+2*a[0]);
z[3]=(3*a[1]+2*a[0]);
z[4]=(3*a[1]+3*a[0]);
z[5]=(3*a[1]+4*a[0]);//z₅=y

//vertices of σ
z[11]=(4*a[1]+4*a[0]);
z[12]=4*a[1]+3*a[0];

//other corners of hulls of σ,τ
z[7]=4*a[1]-a[0];
z[8]=(2*a[1]+4*a[0]);
z[10]=3*a[1]-a[0];//        other vertex of τ besides z₀

//draw the convex hulls
fill(z[0]--z[8]--z[11]--z[7]--z[10]--cycle,white*.85);
//draw the walls of the building:
for(i=-18;i<30;++i){
  draw((i*a[0])--(i*a[0]+20*a[1]),white*.7);
  draw((i*a[0])--(i*a[0]+20*a[2]),white*.7);
}
for(i=0;i<20;++i){
  draw((i*a[2]-10*a[0])--(i*a[2]+(20+i)*a[0]),white*.7);
}

//clip to rectangular region:
clip((0,2)--(14,2)--(14,8)--(0,8)--cycle);

//draw and label σ and τ
draw("\(\tau\)", z[0]--z[10],SW);
draw("\(\sigma\)",z[11]--z[12],N);

//connect and label points
draw(z[0]--z[1]--z[2]--z[3]--z[4]--z[5]);
dot("\(z_0\)",z[0],S);
dot("\(z_1\)",z[1],S);
dot("\(z_2\)",z[2],S);
dot("\(z_3\)",z[3],SSE);
dot("\(z_4\)",z[4],S);
dot("\(z_5=y\)",z[5],SE);

//label region
label("\(\mathcal{H}(\sigma,\tau)\)",.7*z[7]+.3*z[11],S);
\end{asy}
    \caption{Illustration of Lemma~\ref{lem:point_path}}
    \label{fig:point_path}
  \end{figure}

\begin{proof}
  Let~\(\Apart\) be an apartment containing \(\sigma\) and~\(\tau\).  Since \(y\in\Hull(\sigma,\tau)\) implies \(y\in \Apart\), we may restrict our attention to~\(A\).  If the affine root system underlying~\(\Apart\) is decomposable, then \(\Apart=\prod_{i=1}^n \Apart_i\) with apartments of indecomposable affine root systems~\(\Apart_i\).  Each affine root factors through the projection to~\(\Apart_i\) for some~\(i\).  Hence vertices in~\(\Apart\) are nothing but families of vertices in~\(\Apart_i\) for all~\(i\), and two vertices are adjacent if and only if their \(\Apart_i\)\nb-components are adjacent for all~\(i\); polysimplices in~\(\Apart\) are the same as products of simplices in~\(\Apart_i\).  Moreover, \(\Hull(\sigma,\tau) = \prod_{i=1}^n \Hull(\sigma_i,\tau_i)\) if \(\sigma=\prod_{i=1}^n \sigma_i\) and \(\tau=\prod_{i=1}^n \tau_i\).  Therefore, if we can solve the problem for \(\sigma_i\), \(\tau_i\) and~\(y_i\) in~\(\Apart_i\) for each~\(i\), we can solve it for \(\sigma\), \(\tau\) and~\(y\) in~\(\Apart\).  We may assume without loss of generality that the affine root system of~\(\Apart\) is indecomposable.  Then~\(\Apart\) is a simplicial complex.

  An affine root~\(\affRoot\) of the apartment~\(\Apart\) defines a closed half space
  \[
  \affRoot^\le \defeq \{v\in \Apart\mid \affRoot(v)\le0\}
  \]
  We define \(\affRoot^\ge\) and~\(\affRoot^>\) by the same recipe.  The hull of~\(\sigma\) and~\(\tau\) is the intersection of all~\(\affRoot^\le\) with \(\sigma\cup\tau\subseteq\affRoot^\le\).  Let~\(z\) be a vertex.  If \(y \notin \Hull(\sigma,z)\), then there is an affine root~\(\affRoot\) with \(\sigma,z\subseteq\affRoot^\le\) and \(y\in\affRoot^>\).  Since \(y\) is adjacent to~\(\sigma\), \(y\in\affRoot^>\) implies \(\sigma\subseteq\affRoot^\ge\), so that \(\affRoot|_\sigma=0\).  Hence a vertex~\(z\) satisfies \(y\in\Hull(\sigma,z)\) if and only if \(\affRoot(z)>0\) for all affine roots~\(\affRoot\) with \(\affRoot|_\sigma=0\) and \(\affRoot(y)>0\).  Since \(y\in\Hull(\sigma,\tau)\), the same reasoning shows that for each affine root~\(\affRoot\) with \(\affRoot|_\sigma=0\) and \(\affRoot(y)>0\), there is a vertex~\(z_\affRoot\) of~\(\tau\) with \(\affRoot(z_\affRoot)>0\).  Our first task is to find \(z_0\prec\tau\) with \(y\in\Hull(\sigma,z_0)\).  This is trivial if \(y\prec\sigma\), so that we may assume that~\(y\) does not belong to~\(\sigma\).

  Let \(d=\dim \Apart\).  Since~\(\Apart\) is simplicial, any chamber is bounded by exactly \(d+1\) walls.  Let \(\affRoot_0\), \dots, \(\affRoot_d\) be affine roots such that \(\gamma\defeq \bigcap_{j=0}^d \affRoot_j^\ge\) is a chamber that contains \([\sigma,y]\).  Order them so that \(\affRoot_0(y)>0\) and \(\affRoot_j(y)=0\) for \(j\neq0\), and \(\affRoot_j|_\sigma=0\) if and only if \(j\le k\), where~\(k\) is the codimension of~\(\sigma\).  We now modify this chamber until \(\affRoot_j|_\tau\ge0\) for \(j=1,\dotsc,k\).  If there are~\(j\) in this range and \(z\in\tau\) with \(\affRoot_j(z)<0\), then we replace~\(\gamma\) by \(s_{\affRoot_j}(\gamma)\), where~\(s_{\affRoot_j}\) denotes the reflection at the wall \(\ker\affRoot_j\).  This yields another chamber containing \([\sigma,y]\) because \(\affRoot_j\) vanishes on \([\sigma,y]\).  Each reflection reduces the number of~\(j\) between \(1\) and~\(k\) with \(\affRoot_j|_\tau\not\ge0\) at least by~\(1\).  Finitely many such steps achieve \(\affRoot_j|_\tau\ge0\) for \(j=1,\dotsc,k\).  Since \(\affRoot_0|_\sigma=0\), \(\affRoot_0(y)>0\), and \(y\in\Hull(\sigma,\tau)\), there is a vertex \(z_0\prec\tau\) with \(\affRoot_0(z_0)>0\).  We claim that \(y\in\Hull(\sigma,z_0)\).

  Since the roots \(\affRoot_0\), \dots, \(\affRoot_d\) bound a chamber, any affine root~\(\varaffRoot\) is of the form \(\varaffRoot=\sum_{j=0}^d \lambda_j\affRoot_j\) with either \(\lambda_j\ge0\) for all~\(j\) or \(\lambda_j\le0\) for all~\(j\).  Since \(\affRoot_j|_\sigma\ge0\) for all~\(j\), we have \(\varaffRoot|_\sigma=0\) if and only if \(\lambda_j=0\) for \(j>k\).  Furthermore, \(\varaffRoot(y)>0\) if and only if \(\lambda_0>0\), forcing \(\lambda_j\ge0\) for all~\(j\).  Since \(\affRoot_j|_\tau\ge0\) for \(1\le j\le k\) by construction, we get \(\varaffRoot(z)\ge \lambda_0\affRoot_0(z)>0\).  Therefore, \(y\in\Hull(\sigma,z_0)\).

  Furthermore, if \(z\in\Hull(y,z_0)\), then \(\affRoot_0(z)>0\) because \(\affRoot_0(y)>0\) and \(\affRoot_0(z_0)>0\), and \(\affRoot_j(z)\ge0\) for \(j=1,\dotsc,k\) because \(\affRoot_j(y)\ge0\) and \(\affRoot_j(z_0)\ge0\).  Since \(z_i\in\Hull(y,z_{i-1})\) implies \(z_i\in\Hull(y,z_0)\), we conclude that the property \(y\in\Hull(\sigma,z_i)\) follows from the others.  This remains so in the case \(y\prec\sigma\) excluded above.

  Thus it remains to find a finite sequence of vertices \(z_1\), \dots, \(z_m=y\) such that \(z_{i-1}\) and~\(z_i\) are adjacent and \(z_i\in\Hull(y,z_{i-1})\) for \(i=1,\dotsc,m\).  To construct~\(z_i\) given \(z_{i-1}\neq y\), we consider the geodesic~\(\varphi\) between~\(z_{i-1}\) and~\(y\).  Let~\(\omega\) be the simplex such that \(\varphi(t)\) is an interior point of~\(\omega\) for \(t\in(0,\varepsilon)\) for some \(\varepsilon>0\).  Then \(\omega\in \Hull(y,z_{i-1})\), so that we may let~\(z_i\) be another vertex of~\(\omega\).  Since the passage from~\(z_{i-1}\) to~\(z_i\) decreases the (finite) number of walls that separate~\(z_i\) from~\(y\), this construction will lead to \(z_m=y\) after finitely many steps.
\end{proof}

\begin{remark}
  \label{rem:why_two_geometric_lemmas}
  The adjacency assumption in Lemma~\ref{lem:point_path} is necessary.  In buildings, say, of type~\(\tilde{A}_3\), it can happen that there is no vertex \(z_0\prec\tau\) with \(y\in\Hull(\sigma,z_0)\) although \(y\in\Hull(\sigma,\tau)\).  This is why we only get a sequence of polysimplices in Lemma~\ref{lem:simplex_path}.  This phenomenon cannot occur in \(2\)\nb-dimensional buildings.

  The counterexample involves simplices in a single apartment of type~\(\tilde{A}_3\).  Let \(V\defeq \R^4/\R\cdot(1,1,1,1)\).  The roots on this apartment are the affine maps
  \[
  \affRoot_{ijk}(x_1,x_2,x_3,x_4) \defeq x_i - x_j - k\qquad
  \text{for \(1\le i\neq j\le 4\), \(k\in\Z\)}
  \]
  on~\(V\).  Let
  \[
  x\defeq (0,0,0,0),\qquad
  y\defeq (4,2,2,0),\qquad
  z_1\defeq (4,2,1,0),\qquad
  z_2\defeq (4,3,2,0).
  \]
  The points \(z_1\) and~\(z_2\) are adjacent, that is, there is no affine root~\(\affRoot_{ijk}\) with \(\affRoot_{ijk}(z_1)<0<\affRoot_{ijk}(z_2)\).  The point~\(y\) belongs to the hull \(\Hull(x,[z_1,z_2])\), but neither to \(\Hull(x,z_1)\) nor to \(\Hull(x,z_2)\).  To check this, we compute the hulls.  Let
  \[
  V_+ \defeq \{(x_1,x_2,x_3,x_4)\in V\mid x_0\ge x_1\ge x_2\ge x_3\}.
  \]
  Since \(x,y,z_1,z_2\in V_+\), all hulls are contained in~\(V_+\).
  \begin{align*}
    \Hull(x,z_1) &= \{(x_1,x_2,x_3,x_4)\in V_+\mid
    x_0\le x_1+2\le x_2+3\le x_3+4\},\\
    \Hull(x,z_2) &= \{(x_1,x_2,x_3,x_4)\in V_+\mid
    x_0\le x_1+1\le x_2+2\le x_3+4\},\\
    \Hull(x,[z_1,z_2]) &=
    \begin{multlined}[t]
      \{(x_1,x_2,x_3,x_4)\in V_+\mid
      x_0\le x_1+2\le x_2+3\le x_3+5\\
      \text{and}\ x_0\le x_3+4\}.
    \end{multlined}
  \end{align*}
\end{remark}

After these geometric preparations, we can now reduce~\ref{consistent_simplex_convex} to \ref{consistent_convex} and~\ref{consistent_simplex_commute} in four steps.  First, the second statement in Lemma~\ref{lem:point_path} implies \(\Proj_x\Proj_y\Proj_\tau=\Proj_x\Proj_\tau\) if \(x\) and~\(y\) are adjacent vertices with \(y\in\Hull(x,\tau)\): let~\(z\) be a vertex of~\(\tau\) with \(y\in\Hull(x,z)\); then~\ref{consistent_simplex_commute} yields \(\Proj_\tau=\Proj_z\Proj_\tau\) and~\ref{consistent_convex} yields
\[
\Proj_x\Proj_\tau = \Proj_x\Proj_z\Proj_\tau = \Proj_x\Proj_y\Proj_z\Proj_\tau = \Proj_x\Proj_y\Proj_\tau.
\]

Secondly, we claim that \(\Proj_\tau\Proj_\sigma=\Proj_\tau\Proj_y\Proj_\sigma\) if \(y\in\Hull(\sigma,\tau)\) and~\(y\) is adjacent to~\(\sigma\).  Here we use the sequence of adjacent points \((z_i)\) from Lemma~\ref{lem:point_path}.  The first step yields \(\Proj_{z_{i-1}}\Proj_{z_i}\Proj_\sigma = \Proj_{z_{i-1}} \Proj_\sigma\) and \(\Proj_{z_{i-1}}\Proj_{z_i}\Proj_y  = \Proj_{z_{i-1}} \Proj_y \) because \(z_i\) and~\(z_{i-1}\) are adjacent vertices with \(z_i\in\Hull(y,z_{i-1})\) and \(z_i\in\Hull(\sigma,z_{i-1})\); here we use that \(\Hull(\sigma,z_{i-1})\) contains \(\Hull(y,z_{i-1})\) because \(y\in\Hull(\sigma,z_{i-1})\).  Hence the first step yields
\begin{align*}
  \Proj_\tau\Proj_\sigma
  &= \Proj_\tau\Proj_{z_0}\Proj_\sigma
  = \Proj_\tau\Proj_{z_0}\Proj_{z_1}\Proj_\sigma
  = \dotsb
  = \Proj_\tau\Proj_{z_0}\Proj_{z_1}\dotsm\Proj_{z_{m-1}}\Proj_y\Proj_\sigma
  \\&= \Proj_\tau\Proj_{z_0}\Proj_{z_1}\dotsm\Proj_{z_{m-2}}\Proj_y\Proj_\sigma
  = \dotsb
  = \Proj_\tau\Proj_y\Proj_\sigma.
\end{align*}

Thirdly, we claim that \(\Proj_\sigma\Proj_\tau=\Proj_\sigma\Proj_\omega\Proj_\tau\) if \(\omega\in\Hull(\sigma,\tau)\) and~\(\omega\) is adjacent to~\(\tau\).  Each vertex~\(y\) of~\(\omega\) is adjacent to~\(\tau\), so that \(\Proj_y\) commutes with~\(\Proj_\tau\) by~\ref{consistent_simplex_commute}.  Hence the second step yields
\[
\Proj_\sigma\Proj_\tau = \Proj_\sigma\Proj_y\Proj_\tau = \Proj_\sigma\Proj_{[y,\tau]} = \Proj_\sigma\Proj_\tau\Proj_y
\]
for each vertex~\(y\) of~\(\omega\).  Repeating this argument, we get
\[
\Proj_\sigma\Proj_\tau
= \Proj_\sigma\Proj_\tau\cdot \prod_{y\prec\omega} \Proj_y
= \Proj_\sigma\Proj_\tau\Proj_\omega
= \Proj_\sigma\Proj_{[\omega,\tau]}
= \Proj_\sigma\Proj_\omega\Proj_\tau.
\]

Finally, we use Lemma~\ref{lem:simplex_path} to reduce the general case of~\ref{consistent_simplex_convex} to the third step.  Let \(\omega\in\Hull(\tau,\sigma)\) be arbitrary and choose a sequence of polysimplices \(\tau_0\), \dots, \(\tau_m\) as in Lemma~\ref{lem:simplex_path}.  Then the third step yields
\[
  \Proj_\sigma\Proj_\tau
  = \Proj_\sigma\Proj_{\tau_1}\Proj_\tau
  =\dotsb= \Proj_\sigma\Proj_\omega\Proj_{\tau_{m-1}}\dotsm\Proj_{\tau_1}\Proj_\tau
  =\Proj_\sigma\Proj_\omega\Proj_{\tau_{m-2}}\dotsm\Proj_{\tau_1}\Proj_\tau
  =\dotsb=\Proj_\sigma\Proj_\omega\Proj_\tau.
\]

This finishes the proof of Proposition~\ref{pro:consistent}.

\subsection{Proof of exactness}
\label{sec:proof_resolution}

In this section, we prove Theorem~\ref{the:main_theorem_resolution}.  In fact, the theorem remains valid for the admissible subcomplexes introduced in Definition~\ref{def:admissible}.  We will prove it in that generality.

We first assume that~\(\Sigma\) is finite.  Later, we will reduce infinite~\(\Sigma\) to this special case.  Theorem~\ref{the:support_projection} yields the last assertion,
\[
V \cong \sum_{x\in\Sigma^\circ} \Proj_x(V) \oplus \bigcap_{x\in\Sigma^\circ} \ker \Proj_x.
\]
We still have to prove
\begin{equation}
  \label{eq:homology_finite}
  \Ho_n(\Sigma,\Gamma) \cong
  \begin{cases}
    \sum_{x\in\Sigma^\circ} V_x&\text{for \(n=0\),}\\
    0&\text{for \(n>0\).}
  \end{cases}
\end{equation}
The remaining assertion about cohomology follows by the same argument applied to the opposite category, see Lemma~\ref{lem:opposite_category}.

We prove~\eqref{eq:homology_finite} for all admissible finite subcomplexes by a divide and conquer method.

\begin{lemma}
  \label{lem:divide_and_conquer}
  Let~\(\Sigma\) be a finite admissible subcomplex and assume that it can be decomposed as \(\Sigma = \Sigma_+\cup \Sigma_-\) with admissible \(\Sigma_\pm\) and \(\Sigma_0 = \Sigma_+\cap \Sigma_-\).  If~\eqref{eq:homology_finite} holds for \(\Sigma_+\), \(\Sigma_-\), and~\(\Sigma_0\), then it holds for~\(\Sigma\) as well.
\end{lemma}

\begin{proof}
  The cellular chain complexes for these subcomplexes form an exact sequence
  \[
  \Cell(\Sigma_0) \mono \Cell(\Sigma_+)\oplus\Cell(\Sigma_-) \epi \Cell(\Sigma),
  \]
  which generates a Mayer--Vietoris long exact sequence for their homology groups.  This long exact sequence combined with~\eqref{eq:homology_finite} for \(\Sigma_0\), \(\Sigma_+\) and~\(\Sigma_-\) yields \(\Ho_n(\Sigma)=0\) for \(n\ge2\) and the injectivity of the map \(\Ho_0(\Sigma_0) \to \Ho_0(\Sigma_+)\), so that \(\Ho_1(\Sigma)=0\) as well.  Furthermore, we have a short exact sequence
  \[
  \sum_{x\in\Sigma_0^\circ} V_x \mono
  \sum_{x\in\Sigma_+^\circ} V_x \oplus
  \sum_{x\in\Sigma_-^\circ} V_x \epi
  \Ho_0(\Sigma).
  \]
  Now Corollary~\ref{cor:support_additive} yields
  \[
  \sum_{x\in\Sigma_+^\circ} V_x \cap \sum_{x\in\Sigma_-^\circ} V_x
  = \sum_{x\in\Sigma_0^\circ} V_x,\qquad
  \sum_{x\in\Sigma_+^\circ} V_x + \sum_{x\in\Sigma_-^\circ} V_x
  = \sum_{x\in\Sigma^\circ} V_x.
  \]
  Hence \(\Ho_0(\Sigma)\cong \sum_{x\in\Sigma^\circ} V_x\), so that~\(\Sigma\) verifies~\eqref{eq:homology_finite}.
\end{proof}

Next we consider the special case where~\(\Sigma\) is a single polysimplex, so that the idempotents~\(\Proj_\sigma\) for \(\sigma\in\Sigma\) all commute.  For each subset \(I\subseteq\Sigma^\circ\), let~\(\Proj_I^0\) be the product of~\(\Proj_x\) for \(x\in I\) and~\(1-\Proj_x\) for \(x\notin I\).  Since the idempotents~\(\Proj_x\) commute, this is again an idempotent endomorphism of~\(V\), and its action on \(\Cell(\Sigma)\) commutes with the boundary map.  Since \(V\cong \bigoplus_{I\subseteq\Sigma^\circ} \Proj_I^0(V)\), the chain complex \(\Cell(\Sigma)\) is a resolution of \(\sum_{x\in\Sigma} V_x\) if and only if \(\Proj_I^0\Cell(\Sigma)\) is a resolution of
\[
\Proj_I^0\biggl(\sum_{x\in\Sigma} V_x\biggr) =
\begin{cases}
  \Proj_I^0(V)&\text{if \(I\) is non-empty,}\\
  0&\text{if \(I\) is empty}
\end{cases}
\]
for each subset \(I\subseteq\Sigma^\circ\).  This is clear for empty~\(I\), so that we may assume \(I\neq\emptyset\).

The chain complex \(\Proj_I^0\Cell(\Sigma)\) has a very simple structure: the contribution from a polysimplex~\(\sigma\) is \(\Proj_I^0(V)\) if all vertices of~\(\sigma\) belong to~\(I\), and~\(0\) otherwise.  Hence \(\Proj_I^0\Cell(\Sigma)\) computes the homology of the subcomplex~\(\Sigma_I\) of~\(\Sigma\) spanned by~\(I\) with \emph{constant} coefficients in \(\Proj_I^0(V)\).  This homology agrees with \(\Proj_I^0(V)\) if~\(\Sigma_I\) is contractible.  But what if~\(\Sigma_I\) is not contractible?  Here our consistency conditions enter: we claim that~\(\Sigma_I\) is a face of~\(\Sigma\) or \(\Proj_I^0=0\), so that \(\Proj_I^0(V)=0\) and \(\Proj_I^0\Cell(\Sigma)=0\).  If \(x,y\in I\) and \(z\in\Hull(x,y)\) then \(\Proj_x\Proj_z\Proj_y=\Proj_x\Proj_y\) by Condition~\ref{consistent_simplex_convex}.  Since the idempotents involved commute, this means that \(\Proj_z\ge \Proj_x\Proj_y\), that is, \(1-\Proj_z\) vanishes on the range of \(\Proj_x\Proj_y\).  Hence \(\Proj_I^0=0\) if \(x,y\in I\) and \(z\notin I\).  Thus \(\Proj_I^0\neq0\) forces~\(I\) to be convex, that is, a single face of~\(\Sigma\).  Thus~\eqref{eq:homology_finite} holds if~\(\Sigma\) is a single polysimplex.

If~\eqref{eq:homology_finite} failed for some admissible finite subcomplex~\(\Sigma\), then there would be a minimal such~\(\Sigma\), which we pick.  The previous argument shows that~\(\Sigma\) cannot be a single polysimplex.  Lemma~\ref{lem:divide_and_conquer} shows that we cannot cut~\(\Sigma\) into smaller admissible subcomplexes.  We are going to show that any finite admissible subcomplex that is not a single polysimplex may be cut as in Lemma~\ref{lem:divide_and_conquer}.  This will show that no counterexample to~\eqref{eq:homology_finite} can exist.

Since~\(\Sigma\) is not a single polysimplex, there exists a chamber~\(\Delta\) in an apartment~\(\Apart\), and an affine root~\(\affRoot\) corresponding to a wall of~\(\Delta\), such that~\(\Sigma\) contains both a point \(x_+\in\Delta\) with \(\affRoot(x_+)>0\) and an \(x \in A\) with \(\affRoot(x) <0\).  Let \(\varrho \colon \Buil \to \Apart\) be the retraction centered at~\(\Delta\).  We claim that
\begin{align*}
  \Buil_+ &\defeq \{\sigma\in\Buil\mid a|_{\varrho\sigma}\ge0\},\\
  \Buil_- &\defeq \{\sigma\in\Buil\mid a|_{\varrho\sigma}\le0\},\\
  \Buil_0 \, &\defeq \{\sigma\in\Buil\mid a|_{\varrho\sigma}=0\}
\end{align*}
are convex subcomplexes of~\(\Buil\). Indeed, suppose that \(\Delta_1\) and~\(\Delta_2\) are chambers in~\(\Buil_-\), and consider some gallery between them.  If it contains a chamber in~\(\Buil_+\), then it must cross the wall corresponding to~\(a\) twice, and hence the gallery is not minimal.  The geodesic between two points \(x_1 \in \Delta_1\) and \(x_2 \in \Delta_2\) lies inside the union of all such minimal galleries, and therefore entirely in~\(\Buil_-\).  The same reasoning shows that~\(\Buil_+\) is convex, and \(\Buil_0 = \Buil_+ \cap \Buil_-\).

 Lemma~\ref{lem:convex_admissible} yields that \(\Sigma_?\defeq \Buil_?\cap\Sigma\) for \(?\in\{+,0,-\}\) are admissible subcomplexes of~\(\Sigma\).  Hence Lemma~\ref{lem:divide_and_conquer} applies and leads to a contradiction.  This finishes the proof of Theorem~\ref{the:main_theorem_resolution} for admissible finite subcomplexes~\(\Sigma\).

It remains to reduce the assertions in Theorem~\ref{the:main_theorem_resolution} for infinite~\(\Sigma\) to the finite case.  This requires an increasing filtration of~\(\Buil\) by finite convex subcomplexes~\(B_n\) with \(\bigcup B_n=\Buil\).  For instance, we may let~\(B_n\) be the fixed point subcomplex of~\(\CG_n\) for a decreasing sequence of compact open subgroups~\(\CG_n\) in~\(\G_\LF\) with \(\bigcap \CG_n=\nobreak\{1\}\) (Example~\ref{exa:convex}).  Then \(\Sigma_n \defeq \Sigma\cap B_n\) for \(n\in\N\) is an increasing sequence of finite admissible subcomplexes of~\(\Sigma\) with \(\bigcup \Sigma_n = \Sigma\), and
\[
\Cell(\Sigma) \cong \varinjlim \Cell(\Sigma_n).
\]
If we work with modules, then we can now use the exactness of inductive limits to finish the proof in the homological case very quickly.  The cohomological case requires more work and is understood best in the setting of general Abelian categories, where the arguments in the homological and cohomological case are equivalent by Lemma~\ref{lem:opposite_category}.

The maps \(\Cell(\Sigma_n)\to\Cell(\Sigma_{n+1})\) are split monomorphisms by definition.  Hence \(\Cell(\Sigma)\) is not just a colimit but also a homotopy colimit of the sequence of chain complexes \(\Cell(\Sigma_n)\).  This means that there is an exact sequence of chain complexes
\[
0 \to
\bigoplus_{n\in\N} \Cell(\Sigma_n) \xrightarrow{\Id-S}
\bigoplus_{n\in\N} \Cell(\Sigma_n) \to \Cell(\Sigma) \to 0;
\]
here \(S\) is the shift that embeds the summand \(\Cell(\Sigma_n)\) into \(\Cell(\Sigma_{n+1})\).  This exact sequence of chain complexes induces a long exact homology sequence, which we may rewrite as a short exact sequence
\[
0 \to \varinjlim \Ho_*(\Sigma_n)
\to \Ho_*(\Sigma)
\to \varinjlim\nolimits^1 \Ho_{*-1}(\Sigma_n)
\to 0.
\]
Equation~\eqref{eq:homology_finite} implies that the induced maps \(\Ho_0(\Sigma_n) \to \Ho_0(\Sigma_{n+1})\) are split monomorphisms for all \(n\in\N\), and we have already seen that the homology vanishes in other degrees.  Finally, we use that the derived inductive limit functor vanishes for inductive systems of \emph{split} monomorphisms \(\alpha_n\colon X_n\to X_{n+1}\) because such an inductive limit is equivalent to the coproduct of \(X_{n+1}/X_n\) and coproducts in~\(\Abcat\) are assumed to be exact.  Hence our exact sequence shows that the map \(\varinjlim \Ho_*(\Sigma_n) \to \Ho_*(\Sigma)\) is an isomorphism.  The arguments in the cohomological case are dual.

\section{Serre subcategories of smooth representations}
\label{sec:Serre}

Let~\(R\) be a ring with \(\nicefrac1p\in R\).  For instance, \(R\) may be~\(\Z[\nicefrac1p]\) or a field of characteristic not equal to~\(p\).  We define a Hecke algebra \(\Hecke(\G_\LF,R)\) with coefficients in~\(R\) as in Section~\ref{sec:representations}.  Let \((\Proj_x)_{x\in\Buil^\circ}\) be an equivariant and consistent system of idempotents in \(\Hecke(\G_\LF,R)\), that is, the conditions in Definition~\ref{def:consistent_idempotents} hold in \(\Hecke(\G_\LF,R)\).

Let~\(\Abcat\) be an \(R\)\nb-linear category with \emph{exact} countable inductive limits.  The main example is the category of \(R\)\nb-modules.  The opposite category of \(R\)\nb-modules does not work because its inductive limits correspond to projective limits of modules, which are not exact.

An \emph{\(\Hecke(\G_\LF,R)\)-module in~\(\Abcat\)} is an object of~\(\Abcat\) equipped with a ring homomorphism \(\Hecke(\G_\LF,R) \to \Endo(V)\).  We let \(\Rep\) be the category of \(\Hecke(\G_\LF,R)\)-modules in~\(\Abcat\).  We define \emph{smooth} \(\Hecke(\G_\LF,R)\)-modules in~\(\Abcat\) exactly as in Definition~\ref{def:smooth_module}.

If~\(V\) is an \(\Hecke(\G_\LF,R)\)-module in~\(\Abcat\), then the idempotents \(\Proj_x\) in~\(\Hecke(\G_\LF,R)\) are represented by an equivariant consistent system of idempotents in \(\Endo(V)\), which we still denote by \((\Proj_x)\).  This construction is natural in the formal sense, so that the resulting cosheaf \(\Gamma(V)\) and its cellular chain complex depend functorially on~\(V\).

The exactness of inductive limits in~\(\Abcat\) means that inductive limits of monomorphisms in~\(\Abcat\) are again monomorphisms.  In particular, since the natural maps \(\sum_{x\in\Sigma_\fin^\circ} \Proj_x(V) \to V\) for finite convex subcomplexes \(\Sigma_\fin\subseteq\Sigma\) are monomorphisms by definition, so is the induced map \(\varinjlim_{\Sigma_\fin} \sum_{x\in\Sigma_\fin^\circ} \Proj_x(V)\to V\).  Its image is
\[
\sum_{x\in\Sigma^\circ} \Proj_x(V) \defeq \im \biggl(\bigoplus_{x\in\Sigma^\circ} \Proj_x(V)\to V\biggr).
\]
This is the supremum of \(\{\Proj_x(V)\mid x\in\Sigma^\circ\}\) in the directed set of subobjects of~\(V\).  Hence Theorem~\ref{the:main_theorem_resolution} yields
\[
\Ho_0\bigl(\Sigma,\Gamma(V)\bigr) \cong \sum_{x\in\Sigma^\circ} \Proj_x(V) \subseteq V.
\]

\begin{theorem}
  \label{the:Serre_subcategory}
  Let~\(R\) be a ring with \(\nicefrac1p\in R\) and let \((\Proj_x)_{x\in\Buil^\circ}\) be an equivariant consistent system of idempotents in \(\Hecke(\G_\LF,R)\).  Let~\(\Abcat\) be an \(R\)\nb-linear category with exact countable inductive limits.

  The class \(\Rep(\Proj_x)\) of all \(\Hecke(\G_\LF,R)\)-modules~\(V\) in~\(\Abcat\) with
  \[
  V=\sum_{x\in\Buil^\circ} \Proj_x(V)
  \]
  is a Serre subcategory, that is, it is hereditary for extensions, quotients and subobjects \textup{(}and closed under isomorphism, anyway\textup{)}.  Furthermore, this class is closed under coproducts and hence under arbitrary colimits, and all \(V\in\Rep(\Proj_x)\) are smooth.
\end{theorem}

\begin{proof}
  We abbreviate \(\Subrep\defeq\Rep(\Proj_x)\).  We have \(V\in\Subrep\) if and only if the augmentation map \(\alpha_V\colon \bigoplus_{x\in\Buil^\circ} \Proj_x(V)\to V\) is an epimorphism.  If \(V_1\epi V_2\) is an epimorphism, then so is the induced map \(\bigoplus_{x\in\Buil^\circ} \Proj_x(V_1)\epi \bigoplus_{x\in\Buil^\circ} \Proj_x(V_2)\).  Hence~\(\alpha_{V_2}\) is an epimorphism if~\(\alpha_{V_1}\) is one.  Thus quotients of objects in~\(\Subrep\) remain in~\(\Subrep\).  Similarly, coproducts of objects in~\(\Subrep\) remain in~\(\Subrep\).  Since colimits are quotients of coproducts, this implies that~\(\Subrep\) is closed under arbitrary colimits.

  Let \(V_1\mono V_2\epi V_3\) be an extension of \(\Hecke(\G_\LF,R)\)-modules in~\(\Abcat\).  Then we get an extension
  \[
  \bigoplus_{x\in\Buil^\circ} \Proj_x(V_1)\mono \bigoplus_{x\in\Buil^\circ} \Proj_x(V_2) \epi  \bigoplus_{x\in\Buil^\circ} \Proj_x(V_3)
  \]
  as well.  The Snake Lemma shows that \(\alpha_{V_2}\) is an epimorphism if \(\alpha_{V_1}\) and~\(\alpha_{V_3}\) are.  Thus~\(\Subrep\) is closed under extensions.

  For any \(x\in\Buil^\circ\), there is a compact open subgroup~\(\CG_x\) such that~\(\Proj_x\) is \(\CG_x\)\nb-biinvariant.  Given a finite subcomplex~\(\Sigma\), we let \(\CG_\Sigma \defeq \bigcap_{x\in\Sigma^\circ} \CG_x\).  Then~\(\CG_\Sigma\) acts trivially on \(\sum_{x\in\Sigma^\circ} \Proj_x(V)\).  Since \(\sum_{x\in\Buil^\circ} \Proj_x(V)\) is the inductive limit of such subspaces, any \(\Hecke(\G_\LF,R)\)-module in~\(\Subrep\) is smooth.

  Finally, it remains to show that subobjects of objects in~\(\Subrep\) are again in~\(\Subrep\).  Let \(V_1\mono V_2\epi V_3\) be an extension in~\(\Subrep\).  The augmented cellular chain complexes
  \[
  C_j\defeq \bigl( \Cell \bigl(\Buil,\Gamma(V_j) \bigr) \to V_j \bigr)
  \]
  for \(j=1,2,3\) form an extension of chain complexes \(C_1\mono C_2\epi C_3\) as well because taking the range of an idempotent in~\(\Hecke\) is an exact functor on \(\Rep\).  Theorem~\ref{the:main_theorem_resolution} yields that \(V_j\in\Subrep\) if and only if~\(C_j\) is exact.  Now the long exact homology sequence shows that all three of \(C_1\), \(C_2\) and~\(C_3\) are exact once two of them are.  If \(V_2\in\Subrep\), then \(V_3\in\Subrep\) because~\(\Subrep\) is hereditary for quotients; the two-out-of-three property yields \(V_1\in\Subrep\) as well, that is, \(\Subrep\) is closed under subobjects.
\end{proof}

Let~\(V\) be an \(\Hecke(\G_\LF,R)\)-module.  Then \(\sum_{x\in\Buil^\circ} \Proj_x(V)\subseteq V\) is an \(\Hecke(\G_\LF,R)\)-module as well because it is the image of a morphism between \(\Hecke(\G_\LF,R)\)-modules.  Thus we may define a functor
\[
\Phi \colon \Rep\to\Rep,\qquad
V\mapsto \sum_{x\in\Buil^\circ} \Proj_x(V),
\]
which comes with a natural transformation \(\Phi (V) \to V\).

\begin{proposition}
  \label{pro:subcategory_products_coproducts}
  The functor~\(\Phi\) is a retraction from~\(\Rep\) onto the full subcategory \(\Rep(\Proj_x)\), that is, \(\Phi (V)\in\Rep(\Proj_x)\) for all~\(V\) and the natural map \(\Phi (V)\to V\) is an isomorphism for \(V\in\Rep(\Proj_x)\).  The functor~\(\Phi \) is right adjoint to the embedding functor \(\Rep(\Proj_x)\to\Rep\), that is, the natural map \(\Phi (W)\to W\) induces an isomorphism \(\Hom(V,W) \cong \Hom\bigl(V,\Phi (W)\bigr)\) for all \(V\in\Rep(\Proj_x)\), \(W\in\Rep\).
\end{proposition}

\begin{proof}
  We have \(\Proj_x\bigl(\Phi (V)\bigr) = \Proj_x(V)\) because \(\Phi (V)\subseteq V\) and \(\Proj_x(V)\subseteq \Phi (V)\).  Hence \(\Phi \bigl(\Phi (V)\bigr)\cong \Phi (V)\).  By definition, \(\Phi (V)\cong V\) if and only if \(V\in\Rep(\Proj_x)\).  Thus~\(\Phi\) is a retraction from~\(\Rep\) onto~\(\Rep(\Proj_x)\).  An \(\Hecke(\G_\LF,R)\)-module homomorphism \(V\to W\) between \(V,W\in\Rep\) restricts to a map \(\Phi (V)\to\Phi (W)\) because~\(\Phi \) is a functor.  If \(V\in\Rep(\Proj_x)\), that is, \(V\cong\Phi (V)\), then this means that any \(\Hecke(\G_\LF,R)\)-module homomorphism \(V\to W\) factors through the embedding \(\Phi(W)\to W\), necessarily uniquely.  Thus \(\Hom(V,W) \cong \Hom\bigl(V,\Phi(W)\bigr)\) for all \(V\in\Rep(\Proj_x)\), \(W\in\Rep\).
\end{proof}

We may reformulate the definition of \(\Rep(\Proj_x)\) using a fundamental domain for the \(\G_\LF\)\nb-action on~\(\Buil^\circ\).  Recall that~\(\G_\LF\) acts transitively on the set of chambers of~\(\Buil\) and that any vertex of~\(\Buil\) is contained in a chamber~\(\Delta\).  Therefore, if~\(\Delta\) is a chamber in~\(\Buil\), then any \(\G_\LF\)\nb-orbit on~\(\Buil^\circ\) contains a vertex of~\(\Delta\).  Since \(g\Proj_xg^{-1}= \Proj_{gx}\) for all \(g\in\G_\LF\), we may rewrite
\[
\sum_{x\in\Buil^\circ} \Proj_x(V)
= \sum_{g\in\G_\LF} \sum_{x\in\Delta^\circ} \Proj_{gx}(V)
= \sum_{g\in\G_\LF} g\cdot \biggl(\sum_{x\in\Delta^\circ} \Proj_x(V)\biggr).
\]
Thus \(V\in\Rep(\Proj_x)\) if and only if the subspace \(\sum_{x\in\Delta^\circ} \Proj_x(V)\) generates~\(V\) as an \(\Hecke(\G_\LF,R)\)\nb-module.  If \(\Proj_x=\av{\CG_x}\) for a consistent system of compact open subgroups \((\CG_x)_{x\in\Buil^\circ}\) (see Lemma~\ref{lem:consistent_subgroups}), then \(\Proj_x(V)\) is the subspace of \(\CG_x\)\nb-invariants in~\(V\).  Thus~\(\Rep(\av{\CG_x})\) consists of those representations that are generated by their \(\CG_x\)\nb-invariant vectors for \(x\in\Delta^\circ\).  This is the situation considered in~\cite{Bernstein:Centre}.

If the stabiliser~\(\StabS_\Delta\) operates non-trivially on the vertices of~\(\Delta\), then we do not need all vertices of~\(\Delta\) to generate representations in \(\Rep(\Proj_x)\): a set of representatives for the orbits of~\(\G_\LF\) on~\(\Delta\) suffices.  For instance, if \(\G=\Gl_d(\LF)\), then a single vertex suffices (see Section~\ref{sec:examples_consistent}), and \(\Rep(\Proj_x)\) is the set of all \(\Hecke(\G_\LF,R)\)-modules in~\(\Abcat\) that are generated by the range of~\(\Proj_{[\integ^d]}\), where~\([\integ^d]\) is the vertex in \(\BuilGl\) with stabiliser \(\Gl_d(\integ)\).

Our next goal is to show that \(\Rep(\Proj_x)\) is equivalent to the category of unital \(\support_\Delta\Hecke(\G_\LF,R)\support_\Delta\)-modules for any chamber~\(\Delta\), where~\(\support_\Delta\) denotes the support projection of~\(\Delta\) studied in Section~\ref{sec:support_projections}.

Let~\((\Sigma_n)_{n\in\N}\) be an increasing sequence of finite convex subcomplexes of~\(\Buil\) 
with \( \bigcup_{n \in \N} \Sigma_n = \Buil \), and let \(\support_n\defeq \support_{\Sigma_n}\).  
Then \(\support_n\le \support_{n+1}\) for all \(n\in\N\), that is, 
\(\support_n\support_{n+1}=\support_n = \support_{n+1}\support_n\).  Let
\[
\Hecke(\Proj_x) \defeq \bigcup_{n\in\N} \support_n\Hecke(\G_\LF,R)\support_n.
\]
Since this union is increasing, \(\Hecke(\Proj_x)\) is a subalgebra of~\(\Hecke(\G_\LF,R)\).  By construction, \((\support_n)_{n\in\N}\) is an approximate unit of idempotents in~\(\Hecke(\Proj_x)\).  An \(\Hecke(\Proj_x)\)-module~\(V\) in~\(\Abcat\) is called \emph{smooth} if \(V=\varinjlim \support_n V\), where~\(\support_n V\) denotes the image of~\(\support_n\) as an operator on~\(V\).

\begin{proposition}
  \label{pro:Rep_Proj_as_module_category}
  The category \(\Rep(\Proj_x)\) is isomorphic to the category of smooth \(\Hecke(\Proj_x)\)-modules.
\end{proposition}

\begin{proof}
  Any object~\(V\) of \(\Rep\) is also an \(\Hecke(\Proj_x)\)-module.  By definition, an \(\Hecke(\Proj_x)\)-module is smooth if and only if \(V=\varinjlim \support_n V = \varinjlim \sum_{x\in\Sigma_n^\circ} \Proj_x(V)\), that is, if and only if~\(V\) belongs to \(\Rep(\Proj_x)\).  It remains to show that any smooth \(\Hecke(\Proj_x)\)-module structure extends to an \(\Hecke(\G_\LF,R)\)-module structure.  If \(f\in\Hecke(\G_\LF,R)\) and \(g\in\Hecke(\Proj_x)\), then~\(f\) is supported in some compact subset~\(S\) of~\(\G_\LF\) and \(g\in\support_n\Hecke(\G_\LF,R)\support_n\) for some \(n\in\N\). Choose \(N\ge n\) for which~\(\Sigma_N\) contains \(S \cdot \Sigma_n\) and let~\(\lambda\) denote the left regular representation of \(\Hecke(\G_\LF,R)\). Then \(\im \lambda(f*g) \subseteq \sum_{x\in\Sigma_N^\circ} \im \lambda(\Proj_x) = \support_N\Hecke(\G_\LF,R)\), so that \(f*g = \support_N*f*g = \support_N*f*g*\support_N\) because \(N\ge n\).  Thus~\(f\) is a left multiplier of \(\Hecke(\Proj_x)\).  Similarly, \(f\) is a right multiplier of~\(\Hecke(\Proj_x)\).  Thus any smooth \(\Hecke(\Proj_x)\)-module is a module over \(\Hecke(\G_\LF,R)\) as well.
\end{proof}

\begin{theorem}
  \label{the:Hecke_Proj_Morita_unital}
  Let~\(\Delta\) be a chamber of~\(\Buil\) and let \(\Hecke(\Proj_x)_\Delta\defeq \support_\Delta \Hecke(\G_\LF,R)\support_\Delta\).  The category \(\Rep(\Proj_x)\) is equivalent to the category of unital \(\Hecke(\Proj_x)_\Delta\)-modules.
\end{theorem}

\begin{proof}
  This follows from Proposition~\ref{pro:Rep_Proj_as_module_category} if \(\Hecke(\Proj_x)\) is Morita equivalent to \(\Hecke(\Proj_x)_\Delta\).

  The two-sided ideal in~\(\Hecke(\G_\LF,R)\) generated by~\(\support_\Delta\) contains \(\Proj_{g\sigma} = g\Proj_\sigma g^{-1}\) for all \(g\in\G_\LF\) and \(\sigma\in\Delta\) because \(\Proj_\sigma\le\support_\Delta\).  Hence it contains~\(\support_\Sigma\) for any finite convex subcomplex~\(\Sigma\) of~\(\Buil\) by the formula in Theorem~\ref{the:support_projection} for the support projections.  Thus the two-sided ideal of \(\Hecke(\Proj_x)\) generated by~\(\support_\Delta\) contains the approximate unit~\(\support_n\).  This means that the idempotent~\(\support_\Delta\) is full in \(\Hecke(\Proj_x)\).  
So by \cite{Garcia-Simon:Morita}*{Theorem 2.8} the \(\Hecke(\Proj_x)_\Delta\)-\(\Hecke(\Proj_x)\)-bimodules \(\support_\Delta \Hecke(\Proj_x)\) and \(\Hecke(\Proj_x)\support_\Delta\) yield a Morita equivalence between \(\Hecke(\Proj_x)\) and \(\Hecke(\Proj_x)_\Delta\).
\end{proof}

Now we assume that~\(\Abcat\) has exact countable products in order to study the cohomology of the cellular cochain complex \(\DCell\bigl(\Buil,\hat\Gamma(V)\bigr)\) for an \(\Hecke(\G_\LF,R)\)-module~\(V\) in~\(\Abcat\).  Theorem~\ref{the:main_theorem_resolution} yields
\[
\Ho^0\bigl(\Buil,\hat\Gamma(V)\bigr) \cong \varprojlim \support_n(V).
\]
Since \(\support_n(V)=\support_n\bigl(\Phi(V)\bigr)\), this implies
\[
\Ho^0\bigl(\Buil,\hat\Gamma(V)\bigr) \cong \Ho^0\bigl(\Buil,\hat\Gamma(\Phi(V))\bigr),
\]
so that we may restrict attention to \(V\in\Rep(\Proj_x)\) in the following.  Our description of the cohomology is reminiscent of the roughening functor for \(\Hecke(\G_\LF,R)\)-modules, but the comparison of the two constructions requires an additional assumption:

\begin{lemma}
  \label{lem:cohomology_roughening}
  Assume that for each compact open subgroup \(\CG\subseteq\G_\LF\) there is a finite convex subcomplex \(\Sigma\subseteq\Buil\) with \(\av{\CG} \Phi(V) = \sum_{x\in\Sigma^\circ} \av{\CG}\Proj_x(V)\).  Then \(\Ho^0\bigl(\Buil,\hat\Gamma(V)\bigr)\) is the roughening of~\(\Phi(V) \) as a representation of~\(\G_\LF\).  In particular,
  \[
  \Smooth \bigl( \Ho^0 \bigl(\Buil,\hat\Gamma(V) \bigr) \bigr) \cong \Phi (V) .
  \]
\end{lemma}

\begin{proof}
  Let \((\CG_n)_{n\in\N}\) be a decreasing sequence of compact open subgroups with \(\bigcap \CG_n = \{1\}\).  We may assume that~\(\Sigma_n\) is \(\CG_0\)\nb-invariant for all \(n\in\N\), so that \(\support_n\defeq \support_{\Sigma_n}\) commutes with~\(\CG_0\).  Hence the idempotents~\(\support_n\) and~\(\av{\CG_m}\) commute for all \(n,m\in\N\).  Since~\(\support_n\) is a locally constant function on~\(\G_\LF\), it is \(\CG_M\)-invariant for sufficiently large~\(M\).  This means that there is \(M_0\in\N\) with \(\av{\CG_M}\support_n = \support_n\) for all \(M\ge M_0\).  The assumption in the statement means that for each \(m\in\N\) there is \(N_0\in\N\) with \(\av{\CG_m}\support_N(V) = \av{\CG_m}\support_{N_0}(V)\) for all \(N\ge N_0\).  It follows that the projective systems \(\bigl(\av{\CG_m}\support_n(V)\bigr)_{n,m\in\N}\), \(\bigl(\support_n(V)\bigr)_{n\in\N}\), and \(\bigl(\av{\CG_m}\Phi(V)\bigr)_{m\in\N}\) are all equivalent, where \(\Phi(V) = \sum_{x\in\Buil^\circ} \Proj_x(V) = \varinjlim \support_n(V)\).  This implies the assertion.
\end{proof}

The assumption of the lemma is automatic if~\(V\) is admissible in the sense that \(\av{\CG}V\) is finitely generated for each compact open subgroup~\(\CG\) because the finitely many generators must belong to \(\av{\CG}\support_n(V)\) for some \(n\in\N\).  But this assumption is far from necessary:

\begin{proposition}
  \label{pro:cohomology_roughening_Noetherean}
  Let~\(R\) be a field of characteristic not equal to~\(p\).  The limit \(\support_\infty \defeq \lim_{n\to\infty} \support_n\) exists in the multiplier algebra of \(\Hecke(\G_\LF,R)\) and is a central idempotent, that is,
  \[
  \support_\infty f = f\support_\infty = \lim_{n\to\infty} \support_n f = \lim_{n\to\infty} f\support_n\qquad
  \text{for all \(f\in\Hecke(\G_\LF,R)\),}
  \]
  and the sequences converge in the strong sense of becoming eventually constant.

  For any \(\Hecke(\G_\LF,R)\)-module~\(V\), we have $\support_\infty V = \Phi (V)$ and
  \[
  \Ho_0\bigl(\Buil,\Gamma(V)\bigr) \cong \support_\infty V ,\qquad
  \Ho^0\bigl(\Buil,\hat\Gamma(V)\bigr) \cong \Rough(\support_\infty V),
  \]
  where \(\Rough\) denotes the roughening functor.
\end{proposition}

\begin{proof}
  Let \(\Hecke\defeq \Hecke(\G_\LF,R)\).  Since~\(R\) is a field of characteristic not equal to~\(p\), there is a decreasing sequence of compact open subgroups~\((\CG_m)_{m\in\N}\) with \(\bigcap \CG_m=\{1\}\) for which the unital algebras \(\av{\CG_m}\Hecke\av{\CG_m}\) are Noetherean.  For fields of characteristic~\(0\), this is a result of Joseph Bernstein~\cite{Bernstein:Centre}; for fields of finite characteristic not equal to~\(p\), this is due to Marie-France Vign\'eras~\cite{Vigneras:l-modulaires}*{2.13}.

  We fix \(m\in\N\) and assume, as we may, that~\(\Sigma_n\) is \(\av{\CG_m}\)-invariant.  Since \(\av{\CG_m} \Hecke \av{\CG_m}\) is Noetherean, its submodule \(\bigcup_{n\in\N} \av{\CG_m} \Hecke \av{\CG_m} \support_n \) is finitely generated. That is, there exists \(n\in\N\) such that \(\av{\CG_m} \Hecke \av{\CG_m} \support_n = \av{\CG_m} \Hecke \av{\CG_m} \support_N\) for all \(N\ge n\).  Since \(\av{\CG_m}\), \(\support_n\), and~\(\support_N\) are commuting idempotents, this implies \(\av{\CG_m} \support_n = \av{\CG_m} \support_N\). Therefore \(\support_n * f = \support_N * f\) and \(f * \support_n = f * \support_N\) for all \(N\ge n\) and all \(f \in \av{\CG_m}\Hecke \av{\CG_m}\).  Thus the sequences \((f * \support_n)_{n\in\N}\) and \((\support_n * f)_{n\in\N}\) eventually become constant.  Since~\(m\) is arbitrary, we get a multiplier \(\support_\infty\defeq \lim \support_n\).  It is idempotent because all~\(\support_n\) are idempotent.

  Let \(f\in\Hecke\) and let \(X\defeq \supp f\).  For each \(n\in\N\), there is \(N\ge n\) with \(X(\Sigma_n)\subseteq \Sigma_N\).  Then \(\support_N \ge g\support_ng^{-1}\) for all \(g\in X\) and hence \(\support_N*f*\support_n=f*\support_n\).  Thus \(\support_\infty*f*\support_\infty=f*\support_\infty\).  A similar argument yields \(\support_\infty*f*\support_\infty=\support_\infty*f\).  Thus~\(\support_\infty\) is central.

As already noted in proof of Lemma \ref{lem:cohomology_roughening}, $\support_\infty V = \lim_{n \to \infty} \support_n V = \Phi (V)$.

  Recall that \(\av{\CG_m}\support_\infty=\av{\CG_m}\support_n\le \support_n\) for sufficiently large \(n\in\N\).  Since \(\support_n\in\Hecke\), we also have \(\support_n\le \av{\CG_M}\) for sufficiently large \(M\in\N\), hence \(\support_n\le \av{\CG_M}\support_\infty\).  As a consequence, the inductive systems \(\bigl(\av{\CG}_m\support_\infty(V)\bigr)_{m\in\N}\) and \(\bigl(\support_n(V)\bigr)_{n\in\N}\) are equivalent, so that they have isomorphic direct limits.  By Theorem~\ref{the:main_theorem_resolution}, this yields
  \[
  \Ho_0\bigl(\Buil,\Gamma(V)\bigr)
  \cong \varinjlim \support_n(V)
  \cong \varinjlim \av{\CG_m} \support_\infty(V)
  \cong \Smooth\bigl(\support_\infty(V) \bigr) = \support_\infty (V)
  \]
The assertion about $H^0$ can be proved as in Lemma \ref{lem:cohomology_roughening}.
\end{proof}

\begin{proposition}
  \label{pro:products_Rep_Noetherean}
 Let~\(R\) be a field of characteristic not equal to~\(p\).   Then the subcategory \(\Rep(\Proj_x)\) in the category of \emph{smooth} \(\Hecke(\G_\LF,R)\)-modules in~\(\Abcat\) is closed under smooth direct product and hence under arbitrary smooth limits.  That is, if \((V_i)_{i\in I}\) is a family of objects of \(\Rep(\Proj_x)\), then \(\Smooth\bigl( \prod_{i\in I} V_i\bigr)\) belongs to \(\Rep(\Proj_x)\) as well.
\end{proposition}
Notice that the smoothening of the product is a product in the categorical sense in the subcategory of smooth representations.
\begin{proof}
Since (smooth) limits are subobjects of (smooth) products, it suffices to treat products.  The assertion 
is non-trivial because direct products do no commute with arbitrary direct sums, but only with \emph{finite}
sums.  Let~\(\CG\) be a compact open subgroup.  Since~\(\av{\CG}\Hecke(\G_\LF,R)\av{\CG}\) is Noetherean,
there exists a finite convex subcomplex \(\Sigma\subseteq\Buil\) such that 
\[
\av{\CG}\Phi(V) = \av{\CG} \Hecke (\G_\LF,R) \av{\CG} \Phi (V) \;\; \text{equals }
\sum_{x\in\Sigma^\circ} \av{\CG}\Proj_x (V) = \av{\CG} \Hecke (\G_\LF,R) \av{\CG} u_\Sigma V . 
\]
for all \(V\in\Rep\). (See the proof of Proposition~\ref{pro:cohomology_roughening_Noetherean}.)
Since \(V_i\cong\Phi(V_i)\) for all \(i\in I\) by assumption, we get \(\av{\CG}(V_i) = 
\sum_{x\in\Sigma^\circ} \av{\CG}\Proj_x(V_i)\) for all \(i\in I\).  Therefore
  \begin{multline*}
    \av{\CG}\biggl( \prod_{i\in I} V_i\biggr)
    = \prod_{i\in I} \av{\CG} V_i
    = \prod_{i\in I} \sum_{x\in\Sigma^\circ} \av{\CG} \Proj_x(V_i)
    \\= \sum_{x\in\Sigma^\circ} \prod_{i\in I} \av{\CG} \Proj_x(V_i)
    = \av{\CG} \sum_{x\in\Sigma^\circ} \Proj_x \prod_{i\in I} V_i
    = \av{\CG} \Phi \prod_{i\in I} V_i.
  \end{multline*}
Thus \(\Smooth\bigl(\prod V_i\bigr) = \Phi\bigl( \prod V_i\bigr) = \Phi\circ\Smooth\bigl( \prod V_i\bigr)\), 
that is, \(\Smooth\bigl(\prod V_i\bigr)\in\Rep(\Proj_x)\).
\end{proof}

\section{Towards a Lefschetz character formula}
\label{sec:character_formula}

Let~\(R\) be a field whose characteristic is different from~\(p\).  Let~\(V\) be an \(R\)\nb-vector space and let \(\varrho\colon \G_\LF\to \Aut(V)\) be a finitely generated, smooth, admissible representation of~\(\G_\LF\).  That is, any \(v\in V\) is \(\CG\)\nb-invariant for some compact open subgroup~\(\CG\), the subspace of \(\CG\)\nb-invariant vectors in~\(V\) is finite-dimensional for each compact open subgroup~\(\CG\) of~\(\G_\LF\), and~\(V\) is finitely generated as a module over \(\Hecke(\G_\LF,R)\).  This implies that~\(V\) is generated by its \(\CG\)\nb-invariant vectors for a sufficiently small compact open subgroup \(\CG\subseteq\G_\LF\).  Hence \(V\in\Rep(\Proj_x)\) for a suitable equivariant consistent system of idempotents \(\Proj_x\in\Hecke(\G_\LF,R)\) (see~\cite{Schneider-Stuhler:Rep_sheaves}).  We fix such a system \((\Proj_x)_{x\in\Buil^\circ}\) and consider the associated cosheaf~\(\Gamma(V)\).

Admissibility implies that \(\varrho(f)\in\Endo(V)\) is a finite rank operator for each \(f\in\Hecke(\G_\LF,R)\) and hence has a well-defined trace.  This defines an \(R\)\nb-linear map \(\Hecke(\G_\LF,R)\to R\) called the \emph{character} of~\(\varrho\).  For \(R=\C\) a deep theorem of Harish-Chandra asserts that the character 
is of the form \(f\mapsto \int_{\G_\LF} f(x) \chi_\varrho(x) \,\diff x\) for some locally integrable function~\(\chi_\varrho\) that is locally constant at regular semisimple elements.  Thus the character is not just a distribution but a function defined on regular semisimple elements of~\(\G_\LF\).  The values of this character at regular elliptic elements are computed by Peter Schneider and Ulrich Stuhler in~\cite{Schneider-Stuhler:Rep_sheaves}, using the resolutions described above.  The resulting formula is a Lefschetz fixed point formula for the character because it assembles the character value at a regular elliptic element \(g\in\G_\LF\) from contributions by the fixed points of~\(g\) in the building~\(\Buil\).

Jonathan Korman~\cite{Korman:Character} how to extend this computation to general regular compact elements under an additional assumption, which he could verify in the rank-\(1\)-case.  Theorem~\ref{the:main_theorem_resolution} shows that these assumptions are satisfied in general, so that we can compute the character on all compact elements.  The formula we establish here does not yet apply to non-compact regular elements.  We plan to discuss more general character formulas elsewhere, using suitable compactifications of the building.  Our goal here is more modest.

For each polysimplex \(\sigma\in\Buil\), the cosheaf value \(V_\sigma\defeq \Proj_\sigma(V)\) carries a representation of \(\StabS_\sigma\defeq \{g\in\G_\LF\mid g\sigma=\sigma\}\).  We also allow elements of~\(\StabS_\sigma\) to permute the vertices of~\(\sigma\) and even to change orientation.  The representation of~\(\StabS_\sigma\) is the one that appears in the cellular chain complex and thus involves the orientation character \(\StabS_\sigma \to \{\pm1\}\) in~\eqref{eq:orientation_character}.  We let \(\chi_\sigma\colon \StabS_\sigma\to R\) be the character of the representation of~\(\StabS_\sigma\) on~\(V_\sigma\).

Let~\(\CG\) be a compact open subgroup of~\(\G_\LF\).  We want to compute the restriction of the character~\(\chi_\varrho\) of~\(V\) to~\(\CG\) in terms of the characters~\(\chi_\sigma\) of the representations~\(V_\sigma\).  More precisely, we restrict~\(\chi_\sigma\) to \(\CG\cap\StabS_\sigma\) and then extend it by~\(0\) to~\(\CG\).  Summing up these functions over the \(\CG\)\nb-orbit~\(\CG\sigma\), we get the character of the $\CG$\nb-representation \(\Ind_{\CG\cap\StabS_\sigma}^\CG \Res_{\StabS_\sigma}^{\CG\cap\StabS_\sigma} V_\sigma\).

\begin{proposition}
  \label{pro:Lefschetz_character_formula}
  For each \(f\in\Hecke(\CG)\) there is a finite convex subcomplex~\(\Sigma_0\) in~\(\Buil\) such that
  \[
  \chi_\varrho(f)
  = \int_\CG f(g) \cdot \sum_{\sigma\in\Sigma} (-1)^{\deg\sigma} \chi_\sigma(g) \,\diff g
  = \int_\CG f(g) \cdot \sum_{\twolinesubscript{\sigma\in\Sigma}{g\sigma=\sigma}} (-1)^{\deg\sigma} \chi_\sigma(g) \,\diff g
  \]
  for all \(\CG\)\nb-invariant finite convex subcomplexes \(\Sigma\supseteq \Sigma_0\).
\end{proposition}

\begin{proof}
  Since \(f(V)\) is finite-dimensional, it is contained in \(\sum_{x\in\Sigma_0^\circ} \Proj_x(V)\) for some finite convex subcomplex~\(\Sigma_0\).  Let~\(\Sigma\) be a \(\CG\)\nb-invariant finite convex subcomplex containing~\(\Sigma_0\).  Then \(\Cell\bigl(\Sigma,\Gamma(V)\bigr)\) is a chain complex of \(\CG\)\nb-representations, so that~\(f\) acts on it by chain maps.  Theorem~\ref{the:main_theorem_resolution} implies that \(\Cell\bigl(\Sigma,\Gamma(V)\bigr)\) is a resolution of \(\Ho_0\bigl(\Sigma,\Gamma(V)\bigr) \cong\sum_{x\in\Sigma^\circ} \Proj_x(V)\), which contains the range of~\(f\).  Hence the trace of~\(f\) does not change when we view it as an endomorphism of \(\Ho_0\bigl(\Sigma,\Gamma(V)\bigr)\).  The action of~\(f\) on \(\Cell\bigl(\Sigma,\Gamma(V)\bigr)\) lifts the action of~\(f\) on~\(\sum_{x\in\Sigma^\circ} \Proj_x(V)\).

  The vector space \(\bigoplus_{n\in\N} \Cell[n]\bigl(\Sigma,\Gamma(V)\bigr)\) is finite-dimensional because~\(V\) is admissible and~\(\Sigma\) is finite.  Hence the Euler characteristic
  \[
  \sum_{n=0}^\infty (-1)^n \tr\bigl(f|_{\Cell[n](\Sigma,\Gamma(V))}\bigr)
  = \sum_{n=0}^\infty (-1)^n \int_\CG f(g)\cdot \tr\bigl(g|_{\Cell[n](\Sigma,\Gamma(V))}\bigr)\,\diff g
  \]
  is well-defined and agrees with the trace of~\(f\) on \(\Ho^0\bigl(\Sigma,\Gamma(V)\bigr)\), which agrees with the trace \(\chi_\varrho(f)\) of~\(f\) on~\(V\) by our construction of~\(\Sigma\).  We rewrite the Euler characteristic above using the decomposition \(\Cell\bigl(\Sigma,\Gamma(V)\bigr) = \bigoplus_{\sigma\in\Sigma} \Proj_\sigma(V)\):
  \[
  \chi_\varrho(f)
  = \int_\CG f(g) \sum_{\twolinesubscript{\sigma\in\Sigma}{g\sigma=\sigma}}  (-1)^{\deg\sigma}\tr(g|_{\Proj_\sigma(V)}) \,\diff g
  = \int_\CG f(g) \sum_{\twolinesubscript{\sigma\in\Sigma}{g\sigma=\sigma}} (-1)^{\deg\sigma} \chi_\sigma(g) \,\diff g.\qedhere
  \]
\end{proof}

As a consequence, the restriction of the character to~\(\CG\) is the limit of the functions
\[
\chi_\Sigma(g)\defeq \sum_{\twolinesubscript{\sigma\in\Sigma}{g\sigma=\sigma}} (-1)^{\deg\sigma} \chi_\sigma(g) \,\diff g,
\]
where~\(\Sigma\) runs through the set of finite \(\CG\)\nb-invariant convex subcomplexes of~\(\Buil\).  This formula is unwieldy because we cannot exchange the limit over~\(\Sigma\) and the summation: the cancellation between simplices of different parity is needed for the limit to exist.  Recall that the set of simplices~\(\sigma\) with \(g\sigma=\sigma\) is finite if and only if~\(g\) is an regular elliptic element of~\(\G_\LF\).  In this case, the relevant character formula appears already in~\cite{Schneider-Stuhler:Rep_sheaves}.

\section{Conclusion and outlook}
\label{sec:conclusion}

An equivariant consistent systems of idempotents in the Hecke algebra of a reductive \(p\)\nb-adic group produces a natural cosheaf and a natural sheaf on the building of the group for any representation.  The consistency conditions ensure that the homology and cohomology with these coefficients vanishes except in degree zero, where we get a certain subspace and quotient of the representation we started with.  The representations for which the zeroth homology of this cosheaf agrees with the given representation form a Serre subcategory.  We have used support projections to describe this Serre subcategory as the module category over a suitable corner in the Hecke algebra of the group.  These support projections of convex subcomplexes of the building are also a crucial tool for the homology computation.  They are described by a surprisingly simple formula, which only defines an idempotent element of the Hecke algebra because of the consistency conditions.  Since our homological computations still work for convex subcomplexes of the building, we also get a formula for the values of the character of a representation on regular compact elements, which involves the fixed point subset in the building.  But this formula is still unwieldy because the relevant fixed point subsets are infinite for non-elliptic elements, leading to infinite sums that converge only conditionally.

\end{document}